\documentclass[aop,preprint,noinfoline]{imsart}
\RequirePackage[OT1]{fontenc}
\RequirePackage{amsthm,amsmath}
\RequirePackage[numbers]{natbib}
\RequirePackage[colorlinks,citecolor=blue,urlcolor=blue]{hyperref}
\setattribute{journal}{name}{ }

\startlocaldefs

\usepackage[utf8]{inputenc}
\usepackage[OT1]{fontenc}
\usepackage{enumerate}
\usepackage{xcolor}
\usepackage[top=1in, bottom=1in, left=1in, right=1in]{geometry}
\usepackage{graphicx}
\newcommand{\indep}{\rotatebox[origin=c]{90}{$\models$}}
\usepackage{amsfonts}
\usepackage{mathrsfs}	
\usepackage{amssymb}
\usepackage{dsfont}
\usepackage{bbm}

\theoremstyle{plain}
\newtheorem{theorem}{Theorem}[section]
\newtheorem{lemma}[theorem]{Lemma}
\newtheorem{proposition}[theorem]{Proposition}

\theoremstyle{definition}
\newtheorem{definition}[theorem]{Definition}

\newtheorem{assumption}[theorem]{Assumption}

\newtheorem{remark}[theorem]{Remark}

\numberwithin{equation}{section}
\numberwithin{figure}{section}
\usepackage[nodayofweek]{datetime}

\usepackage[nodayofweek]{datetime}

\newcommand{\Ito}{It$\hat{\text{o}}$} 

\newcommand{\bE}{\mathbb{E}}
\newcommand{\bF}{\mathbb{F}}

\newcommand{\bN}{\mathbb{N}}
\newcommand{\bP}{\mathbb{P}}
\newcommand{\bQ}{\mathbb{Q}}
\newcommand{\bR}{\mathbb{R}}

\def \E {\mathbb{E}}

\newcommand{\cB}{\mathcal{B}}
\newcommand{\cC}{\mathcal{C}}

\newcommand{\cF}{\mathcal{F}}
\newcommand{\cG}{\mathcal{G}}

\newcommand{\cP}{\mathcal{P}}

\newcommand{\cX}{\mathcal{X}}
\newcommand{\cY}{\mathcal{Y}}

\renewcommand{\P}{\mathbb{P}}

 \def\k{\kappa}

\newcommand{\1}{\mathbbm{1}}

\usepackage[colorinlistoftodos,bordercolor=orange,backgroundcolor=orange!20,linecolor=orange,textsize=tiny]{todonotes}

\endlocaldefs
\begin{document}	
 

%
\begin{frontmatter}
	\title{Weak Existence and Uniqueness for McKean-Vlasov SDEs with Common Noise} 
	\runtitle{Weak McKean-Vlasov SDEs with Common Noise}

\begin{aug}
\author{\fnms{William R.P.} \snm{Hammersley}\thanksref{t1,m1}\ead[label=e1]{w.r.p.hammersley@sms.ed.ac.uk}},
\author{\fnms{David} \snm{\v{S}i\v{s}ka}\thanksref{m1}\ead[label=e2]{d.siska@ed.ac.uk}}
\and 
\author{\fnms{{\L}ukasz} \snm{Szpruch}\thanksref{m1,m2}		\ead[label=e3]{l.szpruch@ed.ac.uk}} 

\thankstext{t1}{William Hammersley was supported by The Maxwell Institute Graduate School in Analysis and its Applications, a Centre for Doctoral Training funded by the UK Engineering and Physical Sciences Research Council (grant EP/L016508/01), the Scottish Funding Council, Heriot-Watt University and the University of Edinburgh.}
\runauthor{W.R.P. Hammersley et al.}

\affiliation{School of Mathematics, University of Edinburgh\thanksmark{m1} and The Alan Turing Institute, London\thanksmark{m2}, \printead*{e1,e2,e3}}

	\end{aug}

\begin{abstract}
 This paper concerns the McKean-Vlasov stochastic differential equation (SDE) with common noise. An appropriate definition of a weak solution to such an equation is developed. The importance of the notion of compatibility in this definition is highlighted by a demonstration of its r\^ole in connecting weak solutions to McKean-Vlasov SDEs with common noise and solutions to corresponding stochastic partial differential equations (SPDEs). By keeping track of the dependence structure between all components in a sequence of approximating processes, a compactness argument is employed to prove the existence of a weak solution assuming boundedness and joint continuity of the coefficients (allowing for degenerate diffusions). Weak uniqueness is established when the private (idiosyncratic) noise's diffusion coefficient is non-degenerate and the drift is regular in the total variation distance. This seems sharp when one considers using finite-dimensional noise to regularise an infinite dimensional problem. The proof relies on a suitably tailored cost function in the Monge-Kantorovich problem and representation of weak solutions via Girsanov transformations.
 \end{abstract}



\end{frontmatter}
\vspace{-5mm}
{\hypersetup{linkcolor=black}
}


 \section{Introduction}
 
 Distribution dependent stochastic differential equations have been the subject of extensive study since the paper of McKean \cite{mckean1966class}, who was inspired by Kac's foundations of kinetic theory \cite{Kac56}. These equations arise as the limiting behaviour of a representative particle from a mean-field interacting particle system as the number of particles tends to infinity. An introduction to the topic can be found in the notes of Sznitman \cite{sznitman1991topics}. In the case where there is a common noise influencing the individual particles, this correlation gives rise to a form of McKean-Vlasov stochastic differential equation (SDE) with conditioned non-linearity, referred to here as the McKean-Vlasov SDE with common noise. This equation describes the dynamics of a \emph{single} representative particle from the infinite system and is the focus of this paper.
 
 Throughout, let $I:=\mathbb{R}^+$. Given a stochastic process $X$ and a time $T\in I$, the process $X$ stopped at time $T$ will be denoted $X_{\cdot\wedge T}:=\{X_{t\wedge T}\}_{t\in I}$. Let the filtration generated by $X$ be denoted as $\bF^X:=\{\cF^X_t\}_{t\in I}$. Given a probability space supporting a random element $Y$ and a sub-sigma algebra $\cG$, let the regular conditional distribution of $Y$ given $\cG$, should it exist, be written $\mathscr L(Y|\cG)$. Henceforth, let $X$ denote an $\mathbb{R}^{d_X}$-valued stochastic process and let $\mu$ denote a stochastic process valued on the space of probability measures on the path space of $X$. Additionally, $\xi$ will be an $\mathbb{R}^{d_X}$-valued random vector and processes $B$ and $W$ are assumed to be Brownian motions of dimension $d_B$ and $d_W$, respectively. The stochastic inputs $B,W$ and $\xi$ are assumed to be mutually independent. The following system will be referred to as the McKean-Vlasov SDE with common noise: 
 \begin{equation}\label{eq:MKVSDECN}
 \begin{split}
 X_t  =  & \xi  +\int_{0}^{t}b(s,X_{\cdot\wedge s},\mu_s)\, ds +\int_0^t \sigma(s,X_{\cdot\wedge s},\mu_s)\, dW_s+\int_{0}^{t} \rho(s,X_{\cdot\wedge s},\mu_s)\,dB_s,\\
 \mu_s= & \mathscr L  (X_{\cdot\wedge s}|\cF_s^{B,\mu}).\\
 \end{split}
 \end{equation}
 At first sight, the equation satisfied by the random measure flow $\mu$ seems strange, however, should $\mu$ be adapted to $B$, the measure flow satisfies $\mu_s=\mathscr L (X_{\cdot\wedge s}|\cF_s^{B})$ and \eqref{eq:MKVSDECN} takes its more often seen form. 
  Let $\mathcal C$ denote $C(I;\mathbb R^{d_X})$ equipped with the topology of uniform convergence on compact time intervals and $\cP(\cC)$ denote the set of Borel probability measures on $\cC$ equipped with the topology of weak convergence. Finally, let $b,\sigma$ and $\rho$ be measurable functions from $I \times \cC \times \cP(\cC)$ into $\mathbb R^{d_X},\mathbb R^{d_X\times d_W}$ and $\mathbb R^{d_X\times d_B}$, respectively, that are always assumed to be at least \emph{progressive}. To clarify, a function $f$ on $I\times \cC\times \cP(\cC)$ is called \emph{progressive} if for any $t\in I$, $$f(t,x,m)=f(t,x_{\cdot\wedge t},m\circ\phi_t^{-1}),\text{ where }\phi_t:\cC\ni x\mapsto x_{\cdot\wedge t}\in \cC.$$ 
 Of particular importance when working with a common noise are the notions of immersion and compatibility, which are recalled in the following definition. The reader is referred to \cite{carmona2017probabilistic,kurtz2014WeakandStrong,lackerdensesets} for more on these concepts and Appendix \ref{sec immcomp} for some equivalent conditions.
 \begin{definition}[Immersion and Compatibility]
 	Let two filtrations $\bF$ and $\mathbb G$ on a probability space $(\Omega,\cF,\bP)$ be such that $\bF\subset  \mathbb G$. Then $\mathbb F$ is said to be immersed in $\mathbb G$ under $\bP$ if every square integrable $\mathbb F$ martingale is a $\mathbb G$ martingale. For two stochastic processes $X$ and $Y$ defined on this probability space, $X$ is said to be compatible with $Y$ if $\bF^{Y}$ is immersed in $\bF^{X,Y}:=\bF^{X}\vee \bF^{Y}$ under $\P$.
 \end{definition}
 
 Given a measure $\mu$ and an integrable function $f$, let $\langle\mu,f\rangle:=\int f d\mu$. Under appropriate \textit{compatibility} conditions and further specialisation of the coefficients $b,$ $\sigma$ and $\rho$ it will be demonstrated that weak solutions to \eqref{eq:MKVSDECN} yield measure valued solutions to the following SPDE that are both analytically and probabilistically weak. Analytically weak means that the solution is defined via its action on test functions and their derivatives. Probabilistically weak means that the measure valued solution process is not necessarily adapted to the stochastic input (a Brownian motion in this case). 
 The SPDE solved is given as: $\P$-a.s. for all $t\in I$ and all $\varphi \in C^2_b(\mathbb R^{d_X})$
 \begin{equation}
 \label{eq:SPDE}
 \langle  \nu_t, \varphi\rangle =\langle \nu_0,\varphi\rangle+ \int_0^t \langle  \nu_s, L\varphi (s,\cdot,\nu_s)\rangle \, ds  +  \int_0^t \langle \nu_s, \partial_x \varphi \rho(s,\cdot,\nu_s) \rangle\,dB_s,\,\,
 \end{equation}
 where $C^2_b(\mathbb R^{d_X})$ is the set of real valued functions on $\mathbb R^{d_X}$ with continuous and bounded mixed derivatives up to second order. Further, $\partial_x\varphi$ denotes the vector of first order derivatives of $\varphi$ with respect to the components of $x$ and the operator $L$ acts on $C^2_b(\mathbb R^{d_X})$ test functions as follows:
 $$L\varphi (t,x,\mu):= b(t,x,\mu)\partial_x\varphi+\frac{1}{2}\text{trace}((\sigma\sigma^T+\rho\rho^T)(t,x,\mu)\partial_{xx}^2\varphi),$$
 where $\partial^2_{xx}\varphi$ is the matrix of mixed second order derivatives with respect the components of $x$. 
\paragraph{First Key Result: See Theorem \ref{thm spdecorrespondence}} Assume that the coefficients $b$, $\sigma$ and $\rho$ are bounded and \emph{Markovian} in the sense that $(b,\sigma,\rho)(t,x,m)=(b,\sigma,\rho)(t,x_t,m\circ\psi_t^{-1})$ where $\psi_t:\cC\ni x\rightarrow x_t\in\mathbb R^{d_X}$. Then, the existence of a weak solution (to be defined) to the McKean-Vlasov SDE with common noise implies the existence of a measure valued solution the SPDE \eqref{eq:SPDE}.\newline
 
Motivated by the weak formulation of mean field games with common noise given by Carmona, Delarue and Lacker in \cite{carmona2015mfgcn}, careful definitions of strong and weak solutions are given that facilitate this correspondence.
 In this framework, the statements can be brought in line with the generalisation of the well known equivalence of Yamada-Watanabe given by Kurtz in \cite{kurtz2014WeakandStrong}, justifying the form of the solution definitions. Secondly, this framework enables one to keep track of the dependence structure of approximations. This is key in allowing the use of compactness techniques, which are core to the weak existence result for the McKean-Vlasov SDE with common noise given in this paper:\vspace{-0.5mm}
 \paragraph{Second Key Result: See Theorem \ref{thm wkex}} There exists a weak solution to \eqref{eq:MKVSDECN} of the type given in Definition \ref{def weaksoln} under assumptions of boundedness and joint continuity of the coefficients and integrability of the initial vector $\xi$. \vspace{3mm}\newline
The above theorem can be used to help establish an existence result for a particular class of coefficients: 
 \paragraph{Third Key Result: See Theorem \ref{thm wkexintegrated}} Assuming integrability of the initial condition and that the coefficients are Markovian, satisfy a non-degeneracy condition and their dependence on measure is of a linear integrated form with bounded measurable interaction kernel, the corresponding McKean-Vlasov SDE with common noise has a weak solution. \vspace{3mm}\newline
Strong uniqueness of solutions to the McKean-Vlasov SDE with common noise has been long established under the conditions of monotonicity \cite{dawsonVaillancourt1995} or Lipschitz continuity \cite{kurtzXiong1999}. The final and main contribution of this paper is to shed light on the question of uniqueness when the regularity of the coefficients is relaxed. In a non-degenerate setting, uniqueness in joint law for solutions to the McKean-Vlasov with common noise may be established:
\paragraph{Fourth Key Result: See Theorem \ref{thm jwkuniq}} Assume that the diffusion coefficients $\sigma$ and $\rho$ do not depend upon measure and there exists a unique strong solution to the drift-less equation. Let the private noise coefficient $\sigma$ satisfy a non-degeneracy condition and let $\sigma^{-1}b$ be total variation Lipschitz in the measure argument and bounded. Then, the equation \eqref{eq:MKVSDECN} satisfies uniqueness in joint law.\newline
 
The assumptions in the above result allow for only measurability (progressive) in the path argument of $b$ with the price of non-degeneracy of the private noise coefficient $\sigma$. This extends a weak uniqueness argument employed in the case without common noise \cite{campi2018,jabir2019poc,lackerGirsanov2018,stannat2019,mishura2016existence} to the case with a common noise. This idea of uniqueness proof, recently introduced by Mishura and Veretennikov \cite{mishura2016existence}, relies on representing two solutions by Girsanov Transformations from an intermediary probability space and estimating the total variation between the distribution of two solutions. Here, a particular Monge-Kantorovich problem for the path-distributions of solutions is studied, instead of the total variation distance, utilising a cost function tailored to this setting. 
It is easy to see that there is a non-empty intersection of the family of coefficients satisfying the assumptions of Theorem \ref{thm wkexintegrated} and Theorem \ref{thm jwkuniq} for which joint weak existence-uniqueness holds. 

Recently, there has been renewed interest in equations \eqref{eq:MKVSDECN} and \eqref{eq:SPDE}. A brief summary is presented below. This is roughly separated into two categories. The first category comprises of results related to McKean-Vlasov SDEs with common noise and/or stochastic partial differential equations (SPDEs) and the second includes those regarding Mean-Field Games with common noise.
 
Firstly, in contexts a little different from that of this paper, Barbu, R\"ockner and Russo \cite{barbuRocknerRusso2017} consider a type of stochastic porous media equation and Briand et al. \cite{briandCCH2019} study the problem of forwards and backwards SDEs where the distribution of any solution is constrained in some fashion and they extend their analysis to the common noise setting, where instead the conditional distributions are constrained. For well-posedness of a particular class of the McKean-Vlasov SDE with common noise and the corresponding SPDE, see the paper of Coghi and Gess \cite{gess2019} and see those of Kolokoltsov and Troeva \cite{kolokoltsovTroeva2018mkvcn,kolokoltsovTroeva2017} for the sensitivity of solutions to perturbation of the initial data. For models motivated by application to finance and neuroscience, see Hambly and S{\o}jmark \cite{hamblySojmark2019} and Ledger and S{\o}jmark \cite{ledgerSojmark2018}.
Crisan, Janjigian and Kurtz \cite{crisanJanjigianKurtz2018partRepBoundCNDs} study a class of SPDEs that includes the Stochastic Allen-Cahn equation, extending the earlier work of Kurtz and Xiong \cite{kurtzXiong1999} where strong solutions to an infinite system of mean-field interacting particles driven by correlated noises are connected to strong solutions to a non-linear stochastic partial differential equation (SPDE) via the empirical distribution of the particles. Another approach to studying the types of SPDEs associated to particle systems driven by correlated noises is that of Dawson and Vaillancourt \cite{dawsonVaillancourt1995} who obtain measure-valued solutions of the aforementioned SPDE by studying the limit of empirical distributions to interacting systems of finitely many particles as the particle number increases to infinity.
 
In tandem, the mean field game theoretic framework introduced by Huang, Malham\'e and Caines \cite{huangMalhameCaines2006} and Lasry and Lions \cite{lasryLions2007} has recently been subject to rapid development in the direction of common noise. For general theoretical results pertaining to well-posedness of the infinite player equilibrium and its closeness to the finite player equilibria, see \cite{ahuja2016,carmona2015mfgcn,kolokoltsovTroeva2019mfgcn,kolokoltsovTroeva2019mfgcnstable,lacker2014} and the book of Cardaliaguet, Delarue, Lasry and Lions \cite{cardaliaguet2019master}. To see how the presence of a common noise can restore uniqueness to the mean field game, see the papers of Delarue and Tchuendom \cite{delarue2019restoring,delarueTchuendom2019selection,tchuendom2018}. A substantial introduction to mean field games with common noise can be found in the second volume of the book of Carmona and Delarue \cite{carmona2017probabilistic}.
 The standard McKean-Vlasov setting with no common noise remains a popular field of study, with many new results. To list but a few: \cite{barbuRockner2018NLFPKtoDDSDE}, \cite{bossyJabir2018WPMKVsingularDiff},  \cite{hammersley2018}, \cite{huang2019PathDepDDSDEsingCoeff}, \cite{huangWang2018DDSDEsSingularCoeffs}, \cite{jabin2018inventiones},  \cite{chaudruFrikha2018wellPosedness} and \cite{rocknerZhang2018WPDDSDEsingularDrift}.
 \color{black}
 
 In summary, the key contributions of this paper are as follows: first, an appropriate framework is developed which allows one to study weak solutions of McKean-Vlasov SDEs with common noise and, using the compatibility of solutions, connect them with weak solutions of SPDEs, second, this framework allows the use of compactness arguments to obtain weak solutions to said equations and finally, a weak uniqueness result is obtained by a technique inspired by the method introduced in \cite{mishura2016existence}.
 

 \subsection{Definitions of Solutions}
 
To begin, let $\bF^{B,W,\xi}=\{\cF^{B,W,\xi}_t\}_{t\in I}$ be defined by $\cF^{B,W,\xi}_t:=\cF^B_t\vee\cF^W_t\vee\sigma(\xi)=\sigma(B_s,W_s,\xi;0\leq s\leq t)$ for all $t\in I$ and similarly $\bF^{B,\mu}=\{\cF^{B,\mu}_t\}_{t\in I}:=\{\cF^B_t\vee \cF^\mu_t\}_{t\in I}=\{\sigma(B_s,\mu_s;0\leq s\leq t)\}_{t\in I}$. When dealing with a measure space $(\Omega,\cF)$ equipped with multiple probability measures, say $\{\P^{i}\}_i$, denote the laws induced by a random element $X$ under these measures as $\mathscr L^{i}(X)$. Vector and matrix norms will be denoted as $|\cdot|$ and $L_p$ norms as $|\cdot|_{L_p}$. Consider the following definition of a strong solution to \eqref{eq:MKVSDECN}:
 \begin{definition}[Strong Solution to the McKean--Vlasov SDE with Common Noise]\label{def:strong1}
 	A filtered probability space $(\Omega,\cF,\bF,\P)$ equipped with $\bF$ Brownian motions $B$ and $W$ and initial condition $\xi$, all mutually independent, and an $\bF$ adapted $\mathbb R^{d_X}$ valued process $X$ is said to be a \textit{strong solution} to the McKean-Vlasov SDE with common noise if the following conditions hold:
 	\begin{enumerate}[i)]
 		\item $\P$-a.s. for all $t\in I$, $	\int_0^t (|b|+|\sigma|^2+|\rho|^2)(s,X_{\cdot \wedge s},\mathscr L(X_{\cdot \wedge s}|\cF^{B}_s))\, ds < \infty.$
 		\item $X$ is $\bF^{B,W,\xi}$ \label{adaptednesscondition} adapted.
 		\item $\P$-a.s. for all $t\in I$,
 		\[
 		\begin{split}
 		\hspace{-3.5mm}X_t  =   \xi  +\int_{0}^{t}b(s,X_{\cdot\wedge s},\mathscr L  (X_{\cdot\wedge s}|\cF_s^{B}))\, ds & +\int_0^t \sigma(s,X_{\cdot\wedge s},\mathscr L  (X_{\cdot\wedge s}|\cF_s^{B}))\, dW_s\\
 		& +\int_{0}^{t} \rho(s,X_{\cdot\wedge s},\mathscr L  (X_{\cdot\wedge s}|\cF_s^{B}))\,dB_s.
 		\end{split}
 		\]
 	\end{enumerate}
 	
 \end{definition}
 
 One can view a strong solution to the SDE \eqref{eq:MKVSDECN} as a triple of \textit{stochastic inputs} $(B,W,\xi)$ defined on some probability space and a Borel measurable mapping $F: C(I;\mathbb R^{d_B})\times  C(I;\mathbb R^{d_W})\times \mathbb R^{d_X}\rightarrow\mathbb R^{d_X}$ such that $F$ maps the stochastic inputs $(B,W,\xi)$ to an $\bF^{B,W,\xi}$ adapted stochastic process $X:=F(B,W,\xi)$ (the output) such that $(X,B,W,\xi)$ satisfies \eqref{eq:MKVSDECN}. In the language of Kurtz \cite{kurtz2014WeakandStrong} this is a strong compatible solution.
 
 A guess at a good definition for a weak solution could be to remove the adaptedness requirement \ref{adaptednesscondition}) from the above conditions and then ask that a weak solution should consist of a filtered probability space with the rest of Definition \ref{def:strong1} unchanged. For clarity this is subsequently written (the choice of terminology `weak-strong' will be justified after the definition).
 
 \begin{definition}[Weak-Strong Solution to the McKean--Vlasov SDE with Common Noise]\label{def weakstrongmu}
 	A weak-strong solution to \eqref{eq:MKVSDECN} consists of a filtered probability space $(\Omega,\cF,\bF,\P)$ equipped with $\bF$ Brownian motions $B$ and $W$ and initial condition $\xi$, all mutually independent, along with an $\bF$ adapted $\mathbb R^{d_X}$ valued process $X$ that satisfies the following conditions:
 	\begin{enumerate}[i)]
 		\item $\P$-a.s. for all $t\in I$, $	\int_0^t (|b|+|\sigma|^2+|\rho|^2)(s,X_{\cdot \wedge s},\mathscr L(X_{\cdot \wedge s}|\cF^{B}_s))\, ds < \infty.$
 		\item $\P$-a.s. for all $t\in I$,
 		\[
 		\begin{split}
 		\hspace{-3.5mm}X_t  =   \xi  +\int_{0}^{t}b(s,X_{\cdot\wedge s},\mathscr L  (X_{\cdot\wedge s}|\cF_s^{B}))\, ds & +\int_0^t \sigma(s,X_{\cdot\wedge s},\mathscr L  (X_{\cdot\wedge s}|\cF_s^{B}))\, dW_s\\
 		& +\int_{0}^{t} \rho(s,X_{\cdot\wedge s},\mathscr L  (X_{\cdot\wedge s}|\cF_s^{B}))\,dB_s.
 		\end{split}
 		\]
 	\end{enumerate}
 \end{definition}
 
 There is an unfortunate shortcoming of such a definition. One can construct an example where weak solutions are expected to exist, but there are none of the above type. See counter-example 5.1 in \cite{carmona2015mfgcn}. The issue is that one asks that the flow of conditional distributions $\mu$ from \eqref{eq:MKVSDECN} should be adapted to the filtration generated by $B$ and so whilst the process $X$ might not be adapted to the stochastic inputs, the flow of conditional distributions must be. This justifies the terminology weak-strong. Since it is preferable to define weak solutions in such a way that they can be obtained under conditions comparable to the case without common noise, the relaxation to equation \eqref{eq:MKVSDECN} will be made, justified by the following argument. 
 
 Since measurability is not generally preserved under weak limits, methods for approximating the flow of conditional distributions break down. To expand upon this point, imagine that one is solving a stochastic equation $$\Gamma(Y,Z)=0,\,\,\,Y\sim \nu.$$ The notation $Y\sim\nu$ means that the stochastic input $Y$ has distribution $\nu$. $Z$ is the solution/output. Often, one seeks to solve the above by instead considering a mollified equation $\Gamma^n(Y,Z)=0,\,\,\,Y\sim \nu$ such that $``\Gamma^n\rightarrow \Gamma"$ and $\forall n$ the equation is $strongly$ solvable; i.e. there is a measurable function $F^n$ such that $Z^n:=F^n(Y)$ is a solution. Then, passing to the limit in some sense $``\Gamma^n(Y,Z^n)\rightarrow \Gamma(Y,Z)"$ one hopes to recover a solution to the original equation. 
 
 In the case of compactness arguments (weak existence), one may prove the weak convergence of a subsequence of the joint distributions of approximate solutions $(Y,Z^n)$ and represent the solutions on a another probability space $(\bar \Omega,\bar \cF,\bar \P)$ such that $(\bar Y^n,\bar Z^n)\rightarrow(\bar Y,\bar Z)$ pointwise. Since $(\bar Y^n,\bar Z^n)$ have the same distribution as $(Y,Z^n)$, one gets $F^n(\bar Y^n)=\bar Z^n$. Therefore $\bar Z$ is the pointwise limit of $\bar Y^n$ measurable functions, but unfortunately, $\bar Y^n$ varies along the same limit, and one cannot conclude that there is a measurable function $F$ such that $\bar Z=F(\bar Y)$. In fact, the existence of such a function corresponds to the existence of a strong solution. 
 

 The above observations give motivation to relax the measurability requirement of the regular conditional distribution appearing in the equation \eqref{eq:MKVSDECN}. Rather than asking that the measure argument of the coefficients be a version of $\mathscr L(X_{\cdot\wedge s}|\cF^{B}_s)$, one should instead require that the argument be a flow of measures $\mu$ such that for any $s\in I$, $\mu_s=\mathscr L(X_{\cdot\wedge s}|\cF^{B,\mu}_s)$. This relaxation is natural as, in general, this is the only way of identifying the limiting random measures obtained via weak convergence arguments. 
 
 Compatibility however, is preserved under weak limits when the marginal distribution of the stochastic inputs is fixed (see \cite{lackerdensesets}). Due to this fact and the above motivation of connecting to the SPDE, a compatibility condition is introduced in the following definition.

 \begin{definition}[Weak Solution to the McKean--Vlasov SDE with Common Noise]\label{def weaksoln}
 	A weak solution to the McKean-Vlasov SDE with common noise consists of a filtered probability space $(\Omega,\cF,\bF,\P)$ equipped with $\bF$ Brownian motions $B$ and $W$ and an $\cF_0$ measurable random vector $\xi$, all mutually independent, along with $\bF$ adapted processes $X$ and $\mu$ that are $\mathbb R^{d_X}$ and $\cP(\cC)$ valued respectively, satisfying the following conditions:
 	\begin{enumerate}[i)]
 		\item $\int_0^t (|b(s,X_{\cdot\wedge s},\mu_s)|+|\sigma(s,X_{\cdot\wedge s},\mu_s)|^2+|\rho(s,X_{\cdot\wedge s},\mu_s)|^2\, )ds < \infty$ $\P$-a.s. for all $t\in I$.
 		\item $X$ is compatible with $(B,\mu)$, $(X,\mu)$ is compatible with $(B,W,\xi)$ and for $s,t\in I$ with $s\leq t$, $\sigma(W_r-W_s:s\leq r\leq t)\indep \cF^{B,\mu}_t\vee\cF^X_s$.
 		\item $\mu_t=\mathscr L(X_{\cdot\wedge t}|\cF^{B,\mu}_t)$ for all $t\in I$.
 		\item $\P$-a.s. for all $t\in I$, 
 		\begin{equation}\label{eq:MKVSDECNweak}
 		\begin{split}
 		X_t= \xi+\int_{0}^{t}b(s,X_{\cdot\wedge s},\mu_s)\, ds & +\int_0^t \sigma(s,X_{\cdot\wedge s},\mu_s)\, dW_s  +\int_{0}^{t} \rho(s,X_{\cdot\wedge s},\mu_s)\,dB_s. 
 		\end{split}
 		\end{equation}
 		
 	\end{enumerate}
 \end{definition}
In this definition, there is now a pair of outputs, $(X,\mu)$. As a weak solution, these outputs are allowed to have randomness external to that of the stochastic inputs, $(\xi,B,W)$ (i.e. there is not a priori a Borel function $G$ s.t. $(X,\mu)=G(B,W,\xi)$). Further, see that if condition ii) were removed, it would remain implied that $(X,\mu)$ is compatible with $(B,W,\xi)$ since the processes $B$ and $W$ are assumed to be Brownian in the filtration $\bF$ to which all processes are adapted and $\xi$ is assumed $\cF_0$ measurable. However, as these properties will need to be verified in the existence proof to prove that the limiting Brownian motions remain Brownian in the full filtration (generated by all limit processes), they are kept explicit in the definition.

 To further justify considering the flow of measures $\mu$ as part of the solution pair, or `stochastic outputs', note that it is desirable for the definition of a weak solution to be in accord with the Yamada-Watanabe principle.
 
%
 Consider the solution as a pair $(X,\mu)$. Defining pathwise uniqueness such that for any two weak solutions $(X,\mu,B,W,\xi)$ and $(X',\mu',B,W,\xi)$ defined on the same probability space, $(X,\mu)$ and $(X',\mu')$ are indistinguishable. Then by way of the Yamada-Watanabe generalisation of Kurtz \cite{kurtz2014WeakandStrong}, assuming pathwise uniqueness, $(X,\mu)$ becomes $\bF^{B,W,\xi}$ adapted and therefore, due to the independence structure, one can identify $\mu=\mathscr L(X|\cF^{B})$ and recover a strong solution of Definition \ref{def:strong1}. 
 In keeping with the concept of a strong solution used by Kurtz in \cite{kurtz2014WeakandStrong}, the following simple proposition demonstrates that the notion of weak solution given by Definition \ref{def weaksoln} is appropriate. 
 
 \begin{proposition}
 	A strong solution given by Definition \ref{def:strong1} is equivalent to an $\bF^{B,W,\xi}$ adapted weak solution pair $(X,\mu)$ of Definition \ref{def weaksoln}.
 \end{proposition}
 \begin{proof}
 	
 	%
 	

 	Given a strong solution of the type of Definition \ref{def:strong1}, $(B,W,\xi,X)$, define a measure flow $\mu$ by $\mu_t:=\mathscr L(X_{\cdot\wedge t}|\cF^{B}_t)$. By definition, $(X,\mu,B,W,\xi)$ satisfies equation \eqref{eq:MKVSDECNweak} and the integrability condition. Since $\mu$ is $\bF^{B}$ adapted by construction, one has $\cF^{B,\mu}_t=\cF^{B}_t$ for all $t\in I$. Combining this fact with the $\bF^{B,W,\xi}$ adaptedness of $X$, the conditions of Definition \ref{def weaksoln} are easily verified. For the converse direction, note that the independence of $(W,\xi)$ and $(B,\mu)$ combined with the $\bF^{B,W,\xi}$ adaptedness of $\mu$ implies that $\mu$ is $\bF^{B}$ adapted. This in turn allows one to show that $\mu_t=\mathscr L(X_{\cdot\wedge t}|\cF^{B,\mu}_t)=\mathscr L(X_{\cdot\wedge t}|\cF^{B}_t)$ for all $t\in I$ and the equivalence follows. 
 \end{proof}
 
 Should one wish to obtain a weak solution via compactness arguments, when verifying the compatibility of $X$ with $(B,\mu)$ for the weak limit, it becomes advantageous to work with $\mu_t:=\mathscr L(X_{\cdot\wedge t}|\cF^{B,\mu}_\infty)$ and condition on the whole path of $(B,\mu)$. However, with the condition that $X$ is compatible with $(B,\mu)$ in the sense that $\cF^{X}_s$ is conditionally independent of $\cF^{B,\mu}_t$ given $\cF^{B,\mu}_s$ for any $s\leq t\in I$, there is the following equivalence between characterisations of $\mu$.
 \begin{proposition}\label{prop equalityapproach}
 	Given a filtered probability space $(\Omega,\cF,\bF,\P)$ equipped with continuous adapted processes $X$, $B$ and $\mu$, valued in $\mathbb R^{d_X}$, $\mathbb R^{d_B}$ and $\cP(\cC)$ respectively, the following are equivalent:
 	\begin{enumerate}[i)]
 		\item For all $ t\in I$, $\mu_t=\mathscr L(X_{\cdot \wedge t}|\cF^{B,\mu}_t)$ and $X$ is compatible with $(B,\mu)$
 		\item For all $t\in I$, $\mu_t=\mathscr L(X_{\cdot \wedge  t }|\cF^{B,\mu}_\infty).$ 
 	\end{enumerate}

 \end{proposition}	
 \begin{remark}\label{rem consequence}
 	A consequence of either condition in the above proposition is that for all $s\in I$ and any $t \in I:$ $s\leq t$, $\mu_s=\mathscr L(X_{\cdot\wedge s}|\cF^{B,\mu}_{t})$. This property is proved in the beginning of the second half of the following proof. 
 \end{remark}
 \begin{proof}[Proof of Proposition \ref{prop equalityapproach}]
 	First it is shown that $i)\implies ii)$. Fix $t\in I$ and let $f:\cC \rightarrow \mathbb R$ and  $g:C(I;\mathbb R^{d_B})\times C(I; \cP(\cC))\rightarrow \mathbb R$ 
 	all be bounded and Borel measurable. Then, 
 	\begin{equation}\notag
 	\begin{split}
 	\mathbb E [f(X_{\cdot\wedge t} )g(B,\mu)] & =  \mathbb E[\mathbb E [f(X_{\cdot\wedge t})g(B,\mu)|\cF^{B,\mu}_t]]\\
 	& = \mathbb E[\mathbb E [f(X_{\cdot\wedge t})|\cF^{B,\mu}_t]\mathbb E [g(B,\mu)|\cF^{B,\mu}_t]]\\
 	& = \mathbb E[\langle \mu_t,f\rangle \mathbb E [g(B,\mu)|\cF^{B,\mu}_t]]\\
 	& = \mathbb E[\langle \mu_t,f\rangle g(B,\mu)].\\
 	\end{split}
 	\end{equation}
 	The first equality follows from elementary properties of conditional expectation, the second from compatibility (see \ref{thm compat} condition i), the third from definition of $\mu$ and the fourth from the measurability of the mapping $\mu_t\mapsto\langle \mu_t,f\rangle $ and hence the measurability of $\langle \mu_t,f\rangle$ with respect to the sigma algebra $\cF^{B,\mu}_t$. 
 	
 	Since $f$ and $g$ are arbitrary bounded Borel measurable functions, the above equality holds for indicator functions $\1_F$ and $\1_G$ where $F\in \cB(\cC)$ and $G\in\cB(C(I;\mathbb R^{d_B})\times C(I;\cP(\cC )))$. Noting that $\mu_t$ is $\cF^{B,\mu}_\infty$ measurable, $\mu_t$ satisfies the defining properties of the regular conditional distribution of $X_{\cdot\wedge t}$ given $\cF^{B,\mu}_\infty$. 

 	Now it remains to prove that $ii)\implies i)$. Using the fact that for arbitrary $s\leq t\in I$, $\mu_s$ is $\cF^{B,\mu}_{t}$ measurable for any $s\leq t\in I$, and that for any $E\in \cF^{B,\mu}_{t}$ and $F$ defined as above, 
 	$
 	\E[\1_F(X_{\cdot\wedge s})\1_{E} ]= \E[\mu_s(F)\1_{E} ]$ by definition of $\mu_s$, $\mu_s$ can be identified as a version of the regular conditional distribution of $X_{\cdot\wedge s}$ given $\cF^{B,\mu}_{t}$. I.e. for all $s\in I$ and any $t \in I:$ $s\leq t$, $\mu_s=\mathscr L(X_{\cdot\wedge s}|\cF^{B,\mu}_{t})$.

 	The first claim is immediate. To show compatibility, one needs to demonstrate the conditional independence of $\cF^{X}_t$ from $\cF^{B,\mu}_\infty$ given $\cF^{B,\mu}_t$ (see again \ref{thm compat} condition 1.). For fixed $t\in I$, let $f$ and $g$ be as defined above and another function $h$ be defined the same way as $g$. Then,
 	
 	\begin{equation}\notag
 	\begin{split} 
 	& \mathbb E[\E [f(X_{\cdot\wedge  t}) g(B,\mu)|\cF^{B,\mu}_t] h(B_{\cdot\wedge t},\mu_{\cdot\wedge t})]\\
 	&  = \mathbb E[\E [\mathbb E [f(X_{\cdot\wedge  t})|\cF^{B,\mu}_\infty ] g(B,\mu)|\cF^{B,\mu}_t] h(B_{\cdot\wedge t},\mu_{\cdot\wedge t})]\\
 	& = \mathbb E[\mathbb E[\langle \mu_t,f\rangle g(B,\mu)|\cF^{B,\mu}_t]h(B_{\cdot\wedge t},\mu_{\cdot\wedge t})]\\
 	& = \mathbb E[ \langle \mu_t,f\rangle \mathbb E[g(B,\mu)|\cF^{B,\mu}_t]h(B_{\cdot\wedge t},\mu_{\cdot\wedge t})]\\
 	& = \mathbb E[ \mathbb E [f(X_{\cdot\wedge t})|\cF^{B,\mu}_t ]\mathbb E[g(B,\mu)|\cF^{B,\mu}_t]h(B_{\cdot\wedge t},\mu_{\cdot\wedge t})].\\
 	\end{split}
 	\end{equation}
 	The first and third equalities follow from standard properties of conditional expectation and the second from the definition of $\mu$. Finally, the fourth equality holds due to the observation at the beginning of this part of the proof. The conclusion holds by the uniqueness of conditional expectations.
 \end{proof}
 
 
 
 \subsection{Associated SPDE}
 As mentioned in the introduction, assuming further structure of the coefficients, solutions to the McKean-Vlasov SDE with common noise correspond to measure valued solutions of a non-linear SPDE \eqref{eq:SPDE}. The correspondence will be demonstrated in this subsection.

 \begin{definition}[Weak Solution to the SPDE \eqref{eq:SPDE}]
 	A weak solution to the SPDE \eqref{eq:SPDE} is a filtered probability space $(\Omega,\cF,\bF,\P)$ equipped with an $\bF$ Brownian motion $B$ $\bF$ adapted $\cP(\mathbb R^{d_X})$ valued process $\nu$ satisfying the equation \eqref{eq:SPDE}, i.e. $$\langle  \nu_t, \varphi\rangle =\langle \nu_0,\varphi\rangle+ \int_0^t \langle  \nu_s, L\varphi (s,\cdot,\nu_s)\rangle \, ds  +  \int_0^t \langle \nu_s, \partial_x \varphi \rho(s,\cdot,\nu_s) \rangle\,dB_s$$ $\P$-a.s. for all $t\in I$ and for all test functions $\varphi\in C_b^2 (\mathbb R^{d_X})$. 
 \end{definition}
 \begin{theorem}\label{thm spdecorrespondence} Assume that the coefficients $b$, $\sigma$ and $\rho$ are bounded and \emph{Markovian} in the sense that $(b,\sigma,\rho)(t,x,m)=(b,\sigma,\rho)(t,x_t,m\circ\psi_t^{-1})$ where $\psi_t:\cC\ni x\rightarrow x_t\in\mathbb R^{d_X}$. Then, the existence of a weak solution to the McKean-Vlasov SDE with common noise implies the existence of a weak solution the SPDE \eqref{eq:SPDE}. 
 \end{theorem}

 \begin{proof}[Proof of Theorem \ref{thm spdecorrespondence}]
 	First, for any $\varphi\in C_0^\infty(\mathbb R^{d_X})$, apply \Ito's formula for $\varphi(X_t)$:
 	\begin{equation}
 	\notag
 	\begin{split}
 	\varphi(X_t) = & \, \varphi(X_0)+\int_0^t L\varphi(s,X_{ s},\nu_s)\,ds \\
 	& + \int_0^t\partial_x \varphi(X_{s})\sigma (s,X_{ s},\nu_s)\,dW_s + \int_0^t\partial_x \varphi(X_{s})\rho (s,X_{s},\nu_s)\,dB_s\\
 	\end{split}
 	\end{equation}
 	where $\nu_s:=\mu_s\circ\psi_s^{-1}=\mathscr L(X_{s}|\cF^{B,\mu}_s)$. Next, apply the conditional expectation with respect to $\cF^{B,\mu}_t$ on both sides of the above equality:
 	
 	\begin{equation}
 	\notag
 	\begin{split}
 	\mathbb E[\varphi(X_t)|\cF^{B,\mu}_t] = & \, \E[ \varphi(X_0)| \cF^{B,\mu}_t] +\E\bigg [ \int_0^t L\varphi(s,X_{s},\nu_s)\,ds \bigg |\cF^{B,\mu}_t\bigg ] \\
 	& + \E\bigg [\int_0^t\partial_x \varphi(X_{s})\sigma (s,X_{s},\nu_s)\,dW_s  \bigg |\cF^{B,\mu}_t\bigg ]  + \E \bigg [\int_0^t\partial_x \varphi(X_{s}))\rho (s,X_{s},\nu_s)\,dB_s  \bigg |\cF^{B,\mu}_t\bigg ] \\
 	\end{split}
 	\end{equation}
 	Since $\varphi$ has continuous compactly supported derivatives, and the coefficients $b,\sigma,\rho$ are bounded, the integrands in the above expression are bounded and predictable. Therefore, one can apply the stochastic Fubini's theorem \ref{thm stochfub} to the above stochastic integrals, identifying $\bF^1$ as $\bF^{B,\mu}$, $\bF^2$ as $\bF^{X,B,\mu}$, and $\bF^3$ as $\bF$. 
 	\begin{equation}
 	\notag
 	\begin{split}
 	\langle  \nu_t, \varphi\rangle  = & \langle \nu_0,\varphi\rangle  + \int_0^t \E [L\varphi(s,X_{s},\nu_s)  |\cF^{B,\mu}_s  ]\,ds  +  \int_0^t\E [ \partial_x \varphi(X_{s})\rho (s,X_{s},\nu_s)  | \cF^{B,\mu}_s ] \,dB_s\\
 	 = &  \langle \nu_0,\varphi\rangle+\int_0^t\langle \nu_s,L\varphi(s,\cdot,\nu_s)\rangle ds + \int_0^t\langle \nu_s,\partial_x\varphi\rho(s,\cdot,\nu_s)\rangle dB_s.
 	\end{split}
 	\end{equation}
 \end{proof}  
 
 \begin{definition}
 	A strong solution to the SPDE \eqref{eq:SPDE} is an $\bF^B$-adapted weak solution.
 \end{definition}
 \begin{remark}\label{rem weakstrong} If one can conclude that the flow of measures $\mu$ of a weak solution to the McKean-Vlasov SDE with common noise yields a strong solution to the SPDE, then one has a weak-strong solution of the type of Definition \ref{def weakstrongmu}. This fact is exploited in \cite{gess2019}, where Coghi and Gess establish a well-posedness result for \eqref{eq:SPDE}.  
 \end{remark}

 \section{Weak Existence}

 \subsection{Assumptions}
 
 \begin{assumption}[Coefficients]
 	\label{ass coefficients}
 	Functions $b:I\times \cC \times \mathcal{P}(\cC) \to \mathbb{R}^d$, 
 	$\sigma :I\times \cC \times \mathcal{P}(\cC)\to \mathbb{R}^d\times \mathbb{R}^{d_W}$ and $\rho:I\times\cC\times\cP(\cC)\rightarrow \mathbb R^{d_X}\times \mathbb R^{d_B}$
 	are \emph{progressive} (i.e. for any $t\in I$, $(b,\sigma,\rho)(t,x,m)=(b,\sigma,\rho)(t,x_{\cdot\wedge t},m\circ\phi_t^{-1})$, where $\phi_t:\cC\ni x\mapsto x_{\cdot\wedge t}\in \cC$), bounded and jointly continuous in the last two arguments in the following sense:
 	if $(x_n \rightarrow x, m_n \overset{w}{\rightarrow} m )$ as $n\to \infty$
 	then $(b,\sigma,\rho) (t,x_n,m_n) \to (b,\sigma,\rho)(t,x,m)$ as $n\to \infty$. 
 \end{assumption}
 \begin{assumption}[Initial Condition]\label{ass initial}
 	For fixed $p \in (2,\infty]$, $|\xi|_{L_{p}}<\infty$. 
 	
 \end{assumption}

 \color{black}

 \begin{definition}[Euler-type Approximation Scheme]
 	Let $t^n_i:=\frac{i}{n}$ for $i,n\in \bN$, and define $\kappa_n(t):=t^n_i$ for $t\in[t^n_i,t^n_{i+1})$. The sequence of Euler approximations $X^n$, are defined as strong solutions to the following distribution dependent SDEs constructed on a probability space supporting $W,$ $B$ and $\xi$. For all $n\in\bN$, each $X^n$ satisfies $\P$-a.s. for all $t\in I$, 
 	\begin{equation}\label{eq:eulerscheme}
 	\begin{alignedat}{1}
 	X^n_t & = \xi  +\int_{0}^{t}b(s,X^n_{\cdot\wedge \k_n(s)},\mathscr{L}(X^n_{\cdot\wedge\k_n(s)}|\mathcal{F}^B_{\k_n(s)}))\, ds +\int_0^t \sigma(s,X^n_{\cdot\wedge\k_n(s)},\mathscr{L}(X^n_{\cdot\wedge\k_n(s)}|\mathcal{F}^B_{\k_n(s)}))\, dW_s\\
 	&  +\int_{0}^{t} \rho(s,X^n_{\cdot\wedge\k_n(s)},\mathscr{L}(X^n_{\cdot\wedge\k_n(s)}|\mathcal{F}^B_{\k_n(s)}))\,dB_s,\,\,\, \\
 	\end{alignedat}
 	\end{equation}
 	Such solutions exist and can be constructed directly from the triple $(\xi,B,W).$ Since for any $s\in I$ $X^n_{\cdot\wedge\kappa_n(s)}$ is $\cF^{B,W,\xi}_{\kappa_n(s)}$ measurable, $\mathscr L(X^n_{\cdot\wedge\kappa_n(s)}|\cF^B_{\kappa_n(s)})=\mathscr L(X^n_{\cdot\wedge\kappa_n(s)}|\cF^B_{s})=\mathscr L(X^n_{\cdot\wedge\kappa_n(s)}|\cF^B_{\infty})$.
 \end{definition}
 
 \subsection{Auxiliary Lemmas}
 
 \begin{lemma}[A Priori Estimates]\label{lem apriori}
 	Let Assumptions \ref{ass coefficients} and \ref{ass initial} hold. If $\{X^n\}_{n\in \mathbb N}$ is a (the) sequence of continuous stochastic processes satisfying \eqref{eq:eulerscheme}. Then for any $1\leq q \leq p$ and $T<\infty$,
 	\[
 	\sup_n\E \bigg [\sup_{0\leq t \leq T}|X^n_t|^q\bigg ]<\infty.
 	\]
 	For any $q\geq1 $ and $s,t\in I$ such that $|t-s|\leq 1$,
 	\begin{equation}
 	\mathbb E \bigg [\sup_{s\leq u\leq t} |X^n_u-X^n_s|^{q}\bigg ]\leq c_{q}(t-s)^{\frac{q}{2}}.
 	\end{equation}
 	%
 \end{lemma}
 \begin{proof}
 	Is standard in the literature. See, for example, the proof of Theorem 21.9 in \cite{KallenbergBook2002}. 
 \end{proof}
 These estimate allow one to conclude tightness of the family $\{X^n\}_{n\in\mathbb N}$ by application of the Arzel\`a Ascoli characterisation of compact sets (see for example, problem 2.4.11 Karatzas and Shreve \cite{karatzas2012brownian}) and prove that the family of flows of conditional measures constructed for the Euler approximations have continuous versions that induce a tight family of probability measures in $\cP(C(I;\cP_p(\cC )))$.

 \subsection{Existence Theorem}
 
 \begin{theorem}[Existence of a Weak Solution to McKean-Vlasov SDE with Common Noise]\label{thm wkex}
 	Let Assumptions \ref{ass coefficients} and \ref{ass initial} hold. Then there exists a weak solution to the McKean-Vlasov SDE with common noise. 	
 \end{theorem}
 \begin{proof}
There exists a filtered probability space $(\Omega,\cF,\bF,\P)$ satisfying the usual conditions, equipped with mutually independent $\bF$ Brownian motions $B$ and $W$ and initial condition $\xi$. Construct the sequence of approximations $X^n$ satisfying the Euler approximation SDE \eqref{eq:eulerscheme}. This construction is carried out iteratively, applying Lemma \ref{lem cndp} on every interval of the approximation (of length $1/n$ for the $n^{th}$ approximation) to ensure that the conditional distributions are valued in $\cP_p(\cC)$. Note that the processes $X^n$ are continuous by construction and are compatible with $(B,W,\xi)$.
 	It will now be demonstrated that the flow of measures $(\mathscr L(X^n_{\cdot\wedge \kappa_n(t)}|\cF^B_{\kappa_n(t)}))_{t\geq 0}$ have continuous $\cP_p(\cC)$ valued versions by 
 	%
 	verifying the conditions of Theorem \ref{thm kcont}. The following holds for any $s,t\in I$ such that $|t-s|\leq 1$: 
 	
 	\begin{equation}\label{eq eulerwest}
 	\begin{split}
 \mathbb E [ W_p(\mathscr L(X^n_{\cdot\wedge \kappa_n(t)}|\cF^B_{\kappa_n(t)}), \mathscr L(X^n_{\cdot\wedge \kappa_n(s)}|\cF^B_{\kappa_n(t)}))^{p}]  & =   \mathbb E [ W_p(\mathscr L(X^n_{\cdot\wedge t}|\cF^{B}_\infty), \mathscr L(X^n_{\cdot\wedge s}|\cF^B_\infty))^{p}]\\
 	& \leq  \mathbb E [ \mathbb E[\sup_{s\leq u\leq t}|X^n_u-X^n_s|^p|\cF^B_\infty ]]\\
 	& \leq \mathbb E [\sup_{s\leq u\leq t}|X^n_u-X^n_s|^{p}]\\
 	& \leq c_{T,p}(t-s)^{\frac{p}{2}}\\
 	\end{split}
 	\end{equation}
 	The equality follows from Proposition \ref{prop equalityapproach} and the inequalities follow consecutively from the definition of $W_p$, Jensen's inequality, properties of conditional expectation and Lemma \ref{lem apriori}.
 	Since $p>2$, there is a continuous modification (labelled $\mu^n$) of each flow of measures via Theorem \ref{thm kcont}. Moreover, by viewing $\xi$ as the constant process $\{\Xi_t:=\xi\}_{t\in I}$, see that $\mathscr L(X^n_{\cdot\wedge 0}|\cF^B_0)=\mathscr L(X^n_{\cdot\wedge 0})=\mathscr L(\Xi)$ is tight in $\cP_p(\cC)$ as a Dirac mass and since the estimate \eqref{eq eulerwest} is uniform in $n$, the family of continuous modifications of the flows $\mu^n$ is tight in $C(I,\cP_p(\mathbb R^{d_X}))$ by application of Theorem \ref{thm kcontn}.  
 	
 	The family of joint distributions $\mathscr L((X^n,\mu^n,B,W))=:\eta^n$ consequently defines a tight family of measures on $\cC\times C(I;\cP_p(\cC))\times C(I;\mathbb R^{d_B})\times C(I;\mathbb R^{d_W})$.
 	By application of Prokhorov's theorem there is a subsequence $\{n_k\}_k $ and a probability measure $\eta$, such that $\eta^{n_k}\overset{w}{\rightarrow}\eta $.
 	
 	Skorokhod's Representation theorem gives the existence of a probability space $(\tilde{\Omega},\tilde{F},\tilde{P})$ on which are defined random elements $\{\tilde Z^{n_k}\}_k$ and $\tilde Z$, valued on the above product space such that $$\tilde Z^{n_k}\equiv(\tilde X^{n_k},\tilde \mu^{n_k},\tilde{B}^{n_k},\tilde{W}^{n_k})\sim \eta^{n_k},\,\,\, \tilde Z\equiv(\tilde{X},\tilde{\mu},\tilde B,\tilde W)\sim \eta\,\,\,$$ 
 	$$\text{and} \,\,\tilde Z^{n_k}\rightarrow \tilde Z\,\,\,\, \tilde{\omega}\text{-surely}.$$ 
 	It is useful to note that independence/compatibility of one random element/process with respect to another is a property of the joint distribution. This fact will be used to verify a few properties of the  constructed processes. Let the filtration $\tilde \bF$ be defined as $\tilde \cF_t:=\sigma(\tilde X_s,\tilde \mu_s,\tilde B_s,\tilde W_s:s\leq t)$. The adaptedness of the $X$ and $\mu$ with respect to this filtration is immediate from the definition. That $\tilde B$ and $\tilde W$ are $\tilde \bF$ Brownian motions will follow from the immersion of their natural filtrations in the filtration $\tilde \bF$ and this will be verified later in the proof.
 	
 	The proof will be concluded once the components of $\tilde Z$, $(\tilde{X},\tilde{\mu},\tilde B,\tilde W)$ have be shown to satisfy items i) to iv) of Definition \ref{def weaksoln} with $\tilde \xi:=\tilde X_0$. Item 1 follows from the boundedness of $b$, $\sigma$ and $\rho$. 
 	
 	For the second item, it is easily checked that $ \sigma(\tilde W_r-\tilde W_s:s\leq r\leq t)\indep \cF^{\tilde B,\tilde \mu}_t\vee\cF^{\tilde X}_s$ (see \cite{billingsley} Theorem 2.8).  To show that $(\tilde X,\tilde \mu)$ is compatible with $(\tilde B,\tilde W,\tilde \xi)$, one needs to demonstrate the conditional independence of $\tilde \cF^{\tilde X,\tilde \mu}_t$ from $\tilde \cF^{\tilde B,\tilde W,\tilde \xi}_\infty$ given $\tilde \cF^{\tilde B,\tilde W,\tilde \xi }_t$. Let $f:C([0,t];\mathbb R^{d_X}\times\cP_p(\cC))\rightarrow \mathbb R$ continuous and bounded, $g:C(I;\mathbb R^{d_B}\times \mathbb R^{d_W})\times \mathbb R^{d_X}\rightarrow \mathbb R$ and $h:C([0,t];\mathbb R^{d_B}\times\mathbb R^{d_W})\times \mathbb R^{d_X}\rightarrow \mathbb R$ measurable and bounded. Let $X|_{[0,t]}$ denote the truncation of a process on $I$ to its realisation on $[0,t]$. By application of Lemma 2.1 from \cite{lackerdensesets},
 	\begin{equation}\notag
 	\begin{split}
 	& \tilde \bE [f((\tilde X,\tilde \mu)|_{[0,t]})(g(\tilde B,\tilde W,\tilde \xi)-\tilde \bE [ g(\tilde B,\tilde W,\tilde \xi)|\cF^{\tilde B,\tilde W,\tilde \xi}_t])h((\tilde B,\tilde W)|_{[0,t]},\tilde \xi)]\\
 	= & \lim_{k\rightarrow \infty} \tilde \bE \bigg [ f((\tilde X^{n_k},\tilde \mu^{n_k})|_{[0,t]})  \\
 	&  \hspace{12mm}\times\bigg (g(\tilde B^{n_k},\tilde W^{n_k},\tilde \xi^{n_k})-\tilde \bE [ g(\tilde B^{n_k},\tilde W^{n_k},\tilde \xi^{n_k})|\cF^{\tilde B^{n_k},\tilde W^{n_k},\tilde \xi^{n_k}}_t]\bigg ) \\
 	&  \hspace{12mm}\times h((\tilde B^{n_k},\tilde W^{n_k})|_{[0,t]},\tilde \xi^{n_k})\bigg ]   \\
 	= & \lim_{k\rightarrow \infty}  \bE \bigg [f( (X^{n_k} , \mu^{n_k})|_{[0,t]})\\
 	&  \hspace{12mm}\times\bigg (g( B^{}, W^{}, \xi)- \bE [ g( B^{}, W^{}, \xi)|\cF^{ B^{}, W^{}, \xi^{}}_t]\bigg )h(( B^{}, W^{})|_{[0,t]}, \xi^{})\bigg ]  \\
 	= & \, 0.\\
 	\end{split}
 	\end{equation}
 	The final equality holds since $\mu^{n_k}$ is a modification of a $\bF^B$ adapted process on the space $(\Omega,\cF,\P)$ and $X^{n_k}$ is a strong solution to the Euler scheme.
 	
 	To see how to apply Lemma 2.1 from \cite{lackerdensesets}, notice that $\tilde \E[g(\tilde B,\tilde W,\tilde \xi)|\cF^{\tilde B,\tilde W,\tilde \xi}_t]$ is by definition $\cF^{\tilde B,\tilde W,\tilde \xi}_t$ measurable and therefore by the Doob-Dynkin lemma (Lemma \ref{lem doobdynkin}) there exists a measurable function $G:C([0,t];\mathbb R^{d_B}\times \mathbb R^{d_W})\times \mathbb R^{d_X}\rightarrow \bR $ such that $G((\tilde B,\tilde W)|_{[0,t]},\tilde \xi)=\tilde \E[g(\tilde B,\tilde W,\tilde \xi)|\cF^{\tilde B,\tilde W,\tilde \xi}_t]$. Since, $(\tilde B,\tilde W,\tilde \xi)$ has the same distribution as $(\tilde B^{n_k},\tilde W^{n_k},\tilde \xi^{n_k})$,
 	\begin{equation}\notag
 	\begin{split}
 	\tilde \E[\tilde \bE [&  g(\tilde B^{n_k},\tilde W^{n_k},\tilde \xi^{n_k})|\cF^{\tilde B^{n_k},\tilde W^{n_k},\tilde \xi^{n_k}}_t]h((\tilde B^{n_k},\tilde W^{n_k})|_{[0,t]},\tilde \xi^{n_k})]\\
 	= & \tilde \E[g(\tilde B^{n_k},\tilde W^{n_k},\tilde \xi^{n_k})h((\tilde B^{n_k},\tilde W^{n_k})|_{[0,t]},\tilde \xi^{n_k})]\\
 	= & \tilde \E[g(\tilde B,\tilde W,\tilde \xi)h((\tilde B,\tilde W)|_{[0,t]},\tilde \xi)]\\
 	& \tilde \E[\tilde \E[g(\tilde B,\tilde W,\tilde \xi)|\cF^{\tilde B,\tilde W,\tilde \xi}_t]h((\tilde B,\tilde W)|_{[0,t]},\tilde \xi)]\\
 	= & \tilde \E[G((\tilde B,\tilde W)|_{[0,t]},\tilde \xi)h((\tilde B,\tilde W)|_{[0,t]},\tilde \xi)]\\
 	= & \tilde \E[G((\tilde B^{n_k},\tilde W^{n_k})|_{[0,t]},\tilde \xi^{n_k})h((\tilde B^{n_k},\tilde W^{n_k})|_{[0,t]},\tilde \xi^{n_k})].\\
 	\end{split}
 	\end{equation}
 	Therefore, the bounded and measurable function $G$ provides a version of the conditional expectation appearing above, and the Lemma $2.1$ from \cite{lackerdensesets} can be applied.


 It will be verified that for all $t\in I$, $\tilde \mu_t=\mathscr L(\tilde X_t|\cF^{\tilde B,\tilde \mu}_{\infty})$. Then, via Proposition \ref{prop equalityapproach}, it holds that $\tilde \mu_t=\mathscr L(\tilde X_t|\tilde \cF^{\tilde B,\tilde \mu}_t)$ for any $t\in I$ and $\tilde X$ is compatible with $(\tilde B,\tilde \mu)$. This verifies item iii) and the outstanding element of item ii). First, note that since $\tilde \mu$ is adapted to $\tilde \bF^{\tilde B,\tilde \mu}$ (the natural filtration of the tuple $\tilde B,\tilde \mu$), all that needs to be verified to show that $\tilde \mu_t=\mathscr L(\tilde X_t|\tilde \cF^{\tilde B,\tilde \mu}_\infty)$ for any $t\in I$ is that for $f:\cC\rightarrow \mathbb R$ and $g:C(I;\mathbb R^{d_B})\times C(I;\cP(\cC))\rightarrow \mathbb R$ continuous and bounded,
 	\begin{equation}\notag
 	\begin{split}
 	\tilde \bE [f(\tilde X_{\cdot\wedge t})g(\tilde B,\tilde \mu)] & = \tilde \bE [\langle \tilde \mu_t,f \rangle g(\tilde B,\tilde \mu)].\\
 	\end{split}
 	\end{equation}
 	It will hold for $f$ and $g$ bounded and measurable by a Lusin's theorem approximation. 
 	The above equation holds since,
 	\begin{equation}
 	\begin{split}
 	\tilde \bE [f(\tilde X_{\cdot\wedge t})g(\tilde B,\tilde \mu)] = &  \lim_{k\rightarrow \infty}\tilde \bE [f(\tilde X^{n_k}_{\cdot\wedge t})g(\tilde B^{n_k},\tilde \mu^{n_k})] \\
 	= &  \lim_{k\rightarrow \infty} \bE [f( X^{n_k}_{\cdot\wedge t})g( B,\mu^{n_k})] \\
 	= &  \lim_{k\rightarrow \infty} \bE [f( X^{n_k}_{\cdot\wedge t})g( B,\mathscr L(X^{n_k}|\cF^{B}_t))] \\
 	= &  \lim_{k\rightarrow \infty} \bE [\mathscr L( X^{n_k}_{\cdot\wedge t}|\cF^{B}_\infty )(f)g( B,\mathscr L(X^{n_k}|\cF^{B}_t))] \\
 	= &  \lim_{k\rightarrow \infty} \bE [\langle \mu^{n_k}_t,f \rangle g( B,\mu^{n_k})] \\
 	= &  \lim_{k\rightarrow \infty} \tilde \bE [\langle \tilde \mu^{n_k}_t,f \rangle g(\tilde  B^{n_k},\tilde \mu^{n_k})] \\
 	= &  \tilde \bE [\langle \tilde \mu_t,f \rangle g(\tilde  B,\tilde \mu)]. \\
 	\end{split}
 	\end{equation}
 	
 	The first and last equalities follow from dominated convergence, the second and sixth from the fact that the joint distribution of $(X^{n_k},B,\mu^{n_k})$ is the same as that of $(\tilde X^{n_k},\tilde B^{n_k},\tilde \mu^{n_k})$, the third and fifth equalities follow from the fact that $\{\mu^{n_k}_t\}_{t\in I}$ is a modification of $\{\mathscr L(X^{n_k}_t|\cF^B_t)\}_{t\in I}$ and the compatibility of $X^{n_k}$ with $B$, the fourth from the tower property of conditional expectation and definition of regular conditional distributions and the adaptedness of $\{\mathscr L(X^{n_k}_t|\cF^B_t)\}_{t\in I}$ to $\bF^{B}$. The convergence of $\langle \tilde \mu^{n_k}_t,f\rangle $ to $\langle \tilde \mu_t,f\rangle $ follows from the fact that $\tilde \mu^{n_k}_t\rightarrow \tilde \mu_t$ $\tilde \P$-a.s. in $(\cP_p(\cC),W_p)$ - see Theorem 6.9 in \cite{villani2009}. 

 	Finally, the equation \eqref{eq:MKVSDECNweak} will hold $\tilde \P$-a.s. for all $t \in I$ due to Lebesgue's dominated convergence theorem and a theorem due to Skorokhod (pg.32 \cite{skorokhod1965}).

 	
 	All items in the definition of a weak solution have been verified and thus the proof is concluded.
 \end{proof}

\subsection{Weak Existence for Bounded Measurable Interaction Kernel}
Armed with Theorem \ref{thm wkex}, it is possible to prove the existence of weak solutions to a particular class of McKean-Vlasov SDEs with common noise, namely where the coefficients are bounded, measurable, non-degenerate, \emph{Markovian} (in the sense that $(b,\sigma,\rho)(t,x,m)=(b,\sigma,\rho)(t,x_t,m\circ\psi_t^{-1})$ where $\psi_t:\cC\ni x\rightarrow x_t\in\mathbb R^{d_X}$) and the dependence on measure is of the linear integrated form (this is sometimes referred to as a mean field interaction of scalar type). Hence, the spatial regularity of the coefficients can be relaxed at the price of a particular form of measure dependence. To be precise, the following assumption on the coefficients is formulated.

\begin{assumption}\label{ass krylovwkex}
The coefficients $b$, $\sigma$ and $\rho$ take the following form:
\begin{equation}
\label{eq integratedcoeffs}
\begin{split}
f(t,x,\nu):=\int \tilde f(t,x_t,y)\nu\circ\psi_t^{-1}(dy),\\
\end{split}
\end{equation}
where $f$ can be replaced with either $b$, $\sigma$ or $\rho$. The functions (interaction kernels) $\tilde b$, $\tilde \sigma$ and $\tilde \rho$ are assumed to be bounded and measurable and, letting  $\Sigma := (\sigma\,\rho)$, 
\begin{equation}
\label{eq nondegen}
\inf_{t,x,\nu}\inf_{\lambda\in\mathbb R^{d_X}:|\lambda|=1} \lambda^T  \Sigma\Sigma^T \lambda >0.
\end{equation}
\end{assumption} 

\begin{theorem}[Weak Existence for Bounded Measurable Interaction Kernel]\label{thm wkexintegrated}
Under Assumption \ref{ass krylovwkex}, the corresponding McKean-Vlasov SDE with common noise has a weak solution. 
\end{theorem}

\noindent\emph{Proof Outline.} Similar to the proof of Mishura and Veretennikov \cite{mishura2016existence} in the case without common noise, here the argument relies on a mollification of the interaction kernels $\tilde b$, $\tilde \sigma$ and $\tilde \rho$. The resulting \emph{mollified} McKean-Vlasov SDEs with common noise have weak solutions by application of Theorem \ref{thm wkex} and the solution processes satisfy the estimates given in Lemma \ref{lem apriori}. Therefore, a weakly convergent subsequence can be extracted from the sequence of joint laws of the approximate solutions. On a probability space given by the Skorokhod Representation Theorem, the limit process can be shown to be a solution to the original, un-mollified McKean-Vlasov SDE with common noise via application of estimates due to Krylov \cite{krylovBookControlledDiffProcs} (Ch.2 Sec.3 Thm.4). 

\begin{proof}
First, the coefficients are mollified by replacing the interaction kernels with kernels $\tilde b^n$, $\tilde \sigma^n$ and $\tilde \rho^n$ that are defined by,
$$\tilde f^n(t,x,y):=n^{2d_X}\zeta(nx,ny)\ast \tilde f(t,x,y),$$
where $\zeta$ is a non-negative smooth function, vanishing for $|x|+|y|>1$, with $\int \zeta(x,y)dxdy=1$.
It is easy to see that the mollified coefficients satisfy the conditions of Theorem \ref{thm wkex} and hence there exist weak solutions $(X^n,\mu^n,B^n,W^n)$ to the McKean-Vlasov SDEs with common noise defined by the mollified coefficients. Since the kernels' bounds are preserved by the mollification, the coefficients of the mollified McKean-Vlasov SDEs with common noise are uniformly bounded and therefore, by a standard procedure, the conclusion of Lemma \ref{lem apriori} holds for this sequence of weak solutions. By the same argument from the proof of Theorem \ref{thm wkex}, one can extract a weakly convergent subsequence of the laws of these solutions. It will be convenient however, to consider another sequence of probability measures that gives access to copies of the solutions that are conditionally independent given $(\mu^n,B^n)$.

Denote the laws of the solutions (with $\xi^i$ hidden inside $X^i$ since $\xi^i=X^i_0$) by $\mathscr L(X^n,\mu^n,B^n,W^n)$. Disintegrate these distributions (see Chapter 10 in volume II of \cite{bogachevMeasureTheoryBook}) into the joint distribution of $(\mu^n,B^n)$ and the conditional distribution of $(X^n,W^n)$ \emph{given} $\mu^n,B^n$. This is written as 
	$$\mathscr L(X^n,W^n,\mu^n,B^n)(dx,dw,d\nu,db)=p_{X,W}^n(dx,dw,\nu,b)\mathscr L(\mu^n,B^n)(d\nu,db).$$
Introducing a new sequence of probability distributions,
$$\pi^n(dx^1,dw^1,dx^2,dw^2,d\nu,db):=\prod_{i=1}^{2} p_{X,W}^n(dx^i,dw^i,\nu,b) \mathscr L(\mu^n,B^n)(d\nu,db) $$ 
and equipping the product space $\cC\times C(I;\mathbb R^{d_W})\times \cC \times C(I;\mathbb R^{d_W}) \times C(I;\cP(\cC))\times C(I;\mathbb R^{d_B})$ with $\pi^n$, the canonical processes $(X,W,\hat X,\hat W,\mu,B)$ yields two weak solutions $(X,W,\mu,B)$ and $(\hat X,\hat W,\mu,B)$ with the property that $(X,W)$ is conditionally independent of $(\hat X,\hat W)$ given $(\mu,B)$. It is easy to see that the sequence $\pi^n$ is also sequentially compact. As before, one extracts a weakly convergence subsequence and applies Skorokhod's Representation Theorem. Then, abusing notation to let $n$ denote the subsequence, on some probability space there exists random elements $\{(X^{n},W^{n},\hat X^{n}, \hat W^{n},\mu^n,B^n)\sim \pi^n\}_n$ and $(X,W,\hat X,\hat W,\mu,B)\sim \pi=\lim_n\pi^n$ such that $(X^{n},W^{n},\hat X^{n},\hat  W^{n},\mu^n,B^n)\rightarrow (X,W,\hat X,\hat W,\mu,B)$ surely. The aim is to show that $(X,W,\mu,B)$ is a weak solution to the un-mollified McKean-Vlasov SDE with common noise. The first three items of Definition \ref{def weaksoln} are verified as in the proof of Theorem \ref{thm wkex}. The final item (that the SDE holds), however, requires additional consideration. It remains to show that
\begin{equation}
\begin{split}\notag
\int_0^tb^n(s,X^n,\mu^n)\,ds& \rightarrow \int_0^tb(s,X,\mu)\,ds,\\
\int_0^t\sigma^n(s,X^n,\mu^n)\,dW^n_s& \rightarrow\int_0^t\sigma(s,X,\mu)\,dW_s\,\,\text{and} \\
\int_0^t\rho^n(s,X^n,\mu^n)\,dB^n_s& \rightarrow \int_0^t\rho(s,X,\mu)\,dB_s\\
\end{split}
\end{equation}
$\P$-a.s. for all $t\in I$, again allowing $n$ to denote the further subsequence taken to obtain this convergence. Consider some $t\in I\cap \bQ$, and the following sequence of estimates:
\begin{equation}\notag
\begin{split}
& \E\bigg [ \bigg|\int_0^tb^n(s,X^n,\mu^n)ds-\int^t_0b(s,X,\mu)ds\bigg |\bigg ]\\
\leq &  \E\bigg [\int_0^t|b^n(s,X^n,\mu^n)ds-b(s,X,\mu)|ds\bigg ]\\
 \leq &  \E\bigg [\int_0^t|b^n(s,X^n,\mu^n)-b^N(s,X^n,\mu^n)|ds\bigg ] + \E\bigg [\int_0^t|b^N(s,X^n,\mu^n)-b^N(s,X,\mu)|ds\bigg ]\\
 & + \E\bigg [\int_0^t|b^N(s,X,\mu)-b(s,X,\mu)|ds\bigg ]\\
 \end{split}
 \end{equation}
 for some $N\in \bN$. Then, by the form of the measure dependence of $b$ and the Tower property,
 \begin{equation}\label{eq sequence}
 \begin{split}
& \E\bigg [\bigg|\int_0^tb^n(s,X^n,\mu^n)ds-\int^t_0b(s,X,\mu)ds\bigg |\bigg ]\\
 \leq &  \E\bigg [\int_0^t\int |\tilde b^n-\tilde b^N|(s,X^n_s,y)\mu^n\circ\psi_s^{-1}(dy)ds\bigg ]  + \E\bigg [\int_0^t|b^N(s,X^n,\mu^n)-b^N(s,X,\mu)|ds\bigg ]\\
& + \E\bigg [\int_0^t\int |\tilde b^N-\tilde b|(s,X_s,y)\mu\circ\psi_s^{-1}(dy)ds\bigg ]\\
\leq &  \int_0^t\E\bigg [\E\bigg [\int |\tilde b^n-\tilde b^N|(s,X^n_s,y)\mu^n\circ\psi_s^{-1}(dy)\bigg |\cF^{B^n,\mu^n}\bigg ]\bigg ]ds  + \E\bigg [\int_0^t|b^N(s,X^n,\mu^n)ds-b^N(s,X,\mu)|ds\bigg ]\\
& + \int_0^t\E\bigg [\E\bigg [\int |\tilde b^N-\tilde b|(s,X_s,y)\mu\circ\psi_s^{-1}(dy)\bigg |\cF^{B,\mu}\bigg ]\bigg ]ds.\\
\end{split}
\end{equation}
The first term in the final line is handled as follows:
\begin{equation}\notag
\begin{split}
& \int_0^t\E\bigg [\E\bigg [\int |\tilde b^n-\tilde b^N|(s,X^n_s,y)\mu^n(dy)\bigg |\cF^{B^n,\mu^n}\bigg ]\bigg ]ds\\
 = & \int_0^t\E\bigg [\iint |\tilde b^n-\tilde b^N|(s,x,y)\mu^n\circ\psi_s^{-1}(dx)\otimes\mu^n\circ\psi_s^{-1}(dy)\bigg ]ds\\
 = & \int_0^t\E[\E[ |\tilde b^n-\tilde b^N|(s,X^n_s,\hat X^n_s)|\cF^{B^n,\mu^n}]]ds\\
 = & \int_0^t\E[|\tilde b^n-\tilde b^N|(s,X^n_s,\hat X^n_s)]ds\\
 \leq & \,\, | \tilde b^n-\tilde b^N  |_{L_{1+2d}}.\\
\end{split}
\end{equation}
The above equalities hold due to the construction of the measures $\pi^n$ and the inequality by application of Theorem 4, Sec.3, Ch.2 of \cite{krylovBookControlledDiffProcs}. 

Repeating the above sequence of estimates with the superscript $n$ removed, the final term of \eqref{eq sequence} can be dealt with leading to the estimate:
\begin{equation}\notag
\begin{split}
& \E\bigg [\bigg |\int_0^tb^n(s,X^n,\mu^n)ds-\int^t_0b(s,X,\mu)ds\bigg |\bigg ]\\
\leq & \,\, | \tilde b^n-\tilde b^N  |_{L_{1+2d}} + \E\bigg [\int_0^t|b^N(s,X^n,\mu^n)ds-b^N(s,X,\mu)|ds\bigg ] +  | \tilde b^N-\tilde b  |_{L_{1+2d}}.\\
\end{split}
\end{equation}
For any $\varepsilon >0$, there is an $N$ large enough such that for $n > N$, $| \tilde b^n-\tilde b^N  |_{L_{1+2d}}+| \tilde b^N-\tilde b  |_{L_{1+2d}}<\varepsilon/2$. Also, as $n\rightarrow \infty$, by the continuity of $b^N$, the middle term in the above inequality vanishes. Therefore, for each $N \in \bN$, there is an $n_N$ such that for all $n>n_N$, the middle term is bounded by $\varepsilon/2$ and therefore, 
$$\int_0^tb^n(s,X^n,\mu^n)ds\overset{\P}{\rightarrow} \int_0^tb(s,X,\mu)ds $$
for any $t\in I\cap \bQ$. This can be elevated to almost sure convergence along a subsequence and to all $t\in I$ by continuity. To prove the corresponding limits for the stochastic integrals, one follows an analogous procedure to that of the drift convergence. Writing $f,M$ in place of $\sigma,W$ or $\rho, B$, one can estimate as follows:
 \begin{equation}\label{eq stochconv}
 \begin{split}
 & 1/3\E\bigg [ \bigg (\int_0^tf^n(s,X^n,\mu^n)dM^n_s-\int^t_0f(s,X,\mu)dM_s\bigg )^2\bigg ]\\
 \leq &  \E\bigg [ \bigg (\int_0^t(f^n(s,X^n,\mu^n)-f^N(s,X^n,\mu^n))dM^n_s\bigg )^2\bigg ]\\
 & + \E\bigg [ \bigg (\int_0^tf^N(s,X^n,\mu^n)dM^n_s-\int_0^tf^N(s,X,\mu)dM_s\bigg )^2\bigg ]\\
 & +  \E\bigg [ \bigg (\int_0^t(f^N(s,X,\mu)-f(s,X,\mu))dM_s\bigg )^2\bigg ]\\
 \end{split}
 \end{equation}
 for some $N\in \bN$. To finish, apply the It$\hat{\text{o}}$ isometry to the first and third terms on the right hand side of \eqref{eq stochconv} and follow an almost exactly analogous procedure as with the drift convergence, taking care of the second power appearing. Handle the second term with Skorokhod's lemma for the convergence of stochastic integrals, see \cite{skorokhod1965} pg.32. One arrives at the following estimate:
 \begin{equation}\notag
 \begin{split}
 & 1/3\E\bigg [ \bigg (\int_0^tf^n(s,X^n,\mu^n)dM^n_s-\int^t_0f(s,X,\mu)dM_s\bigg )^2\bigg ]\\
 \leq &   | \tilde f^n-\tilde f^N  |_{L_{2(1+2d)}}^2 
  + \E\bigg [ \bigg (\int_0^tf^N(s,X^n,\mu^n)dM^n_s-\int_0^tf^N(s,X,\mu)dM_s\bigg )^2\bigg ] + | \tilde f^N-\tilde f  |_{L_{2(1+2d)}}^2 \\ 
  < & \varepsilon
 \end{split}
 \end{equation}
for sufficiently large $n$ depending on the choice of $\varepsilon>0$. 
\end{proof}

 \section{Uniqueness in Joint Law}
 In this section, a particular class of equations of the type \eqref{eq:MKVSDECNweak} will be studied. Namely, the case where the diffusion coefficients $\sigma$ and $\rho$ do not depend upon measure.
 The authors expect that with similar techniques to those given in \cite{stannat2019} and \cite{mishura2016existence} the result here can be extended to include some spatial growth. However, in the interest of conveying how one overcomes the barriers of extending this method to the common noise setting without become mired in additional technical difficulties, the following assumptions are made regarding the coefficients.
 \begin{assumption}\label{ass uniq}
 	The coefficients $b$, $\sigma$ and $\rho$ are measurable and progressive. The coefficients $\sigma$ and $\rho$ do not depend on the measure argument and are such that there exists a unique strong solution to the driftless SDE:
 	\begin{equation}\label{eq driftless}
 	dX^0_t=\sigma(t,X^0)dW_t+\rho(t,X^0)dB_t.
 	\end{equation}
 	Further, $d_X=d_W$, $\sigma$ is non-degenerate, invertible and $\sigma^{-1} b$ is bounded and Lipschitz continuous in the measure component with respect to the total variation distance, i.e. there is a constant $c_{\text{\tiny TV}}$ such that
 	\begin{equation}\notag
 	|\sigma(t,x)^{-1}b(t,x,\mu)-\sigma(t,x)^{-1}b(t,x,\nu)|\leq c_{\text{\tiny TV}}d_{\text{\tiny TV}}(\mu,\nu).
 	\end{equation}
 \end{assumption}
 Under the above assumption, the McKean-Vlasov SDE with common noise, \eqref{eq:MKVSDECNweak}, takes the form:
 \begin{equation}\label{eq: MKVSDECNuniq}
 X_t=\xi+\int_{0}^{t}b(s,X_{\cdot\wedge s},\mu_s)\, ds +\int_0^t \sigma(s,X_{\cdot\wedge s})\, dW_s+\int_{0}^{t} \rho(s,X_{\cdot\wedge s})\,dB_s.\,\, 
 \end{equation}
 
 \begin{definition}[Uniqueness in Joint Law]
 	The McKean-Vlasov SDE with common noise is said to satisfy `uniqueness in joint law' if any two weak solutions (in the sense of Definition \ref{def weaksoln}), $(X^1,\mu^1,B^1,W^1,\xi^1)$ and $(X^2,\mu^2,B^2,W^2,\xi^2)$ have the same \emph{joint} distribution. 
 \end{definition}
 
 \begin{theorem}\label{thm jwkuniq}
 	Under Assumption \ref{ass uniq}, the McKean-Vlasov SDE with common noise of the form \eqref{eq: MKVSDECNuniq} satisfies uniqueness in joint law.
 \end{theorem}

The proof of Theorem \ref{thm jwkuniq} will be given in Subsection \ref{sec uniqproof}. The following subsection provides a lemma that establishes uniqueness in joint law for the SDEs with random coefficients obtained when one considers the measure valued process provided by a weak solution to \eqref{eq: MKVSDECNuniq} as a stochastic input.

\subsection{Auxiliary Lemma}
\begin{definition}\label{def rsdesoln}
	A filtered probability space supporting Brownian motions $W$ and $B$, an adapted stochastic process $\mu$ and an $\cF_0$ measurable random vector $\xi$, such that $(B,\mu)\indep (W,\xi)$ is said to be a weak solution on $[0,T]$ to the SDE with random coefficients: 
	\begin{equation}\label{eq rsdelem}
	X_t=\xi+\int_0^tb(s,X,\mu)ds+\int_0^t\sigma(s,X)dW_s+\int_0^t\rho(s,X)dB_s,
	\end{equation}
	if it also supports an adapted process $X$, such that
	\begin{enumerate}
		\item $\P$-a.s. $\forall$ $t\in [0,T]$, $\int_0^t |b(s,X,\mu)|+|\sigma(s,X,\mu)|^2+|\rho(s,X ,\mu)|^2\, ds < \infty.$
		\item $X,\mu,B,W,\xi$ satisfy \eqref{eq rsdelem} $\mathbb{P}$-a.s. $\forall$ $t\in [0,T]$.
	\end{enumerate}
\end{definition}
\begin{lemma}
	\label{lem uniqRSDE}
	Under Assumption \ref{ass uniq}, the SDE with random coefficients \eqref{eq rsdelem} satisfies joint uniqueness in law on $[0,T]$ for any $T<\infty$. \newline
	Which is to say that given any two weak solutions of type of Definition \ref{def rsdesoln}, $(\Omega^1,\cF^1,\P^1,X^1,\mu^1,B^1,W^1,\xi^1)$ and $(\Omega^2,\cF^2,\P^2,X^2,\mu^2,B^2,W^2,\xi^2)$ such that $\mathscr L^1(\mu^1,B^1,W^1,\xi^1)=\mathscr L^2(\mu^2,B^2,W^2,\xi^2)$, the joint distributions of the solutions $\mathscr L^1(X^1_{\cdot\wedge T},\mu^1,B^1,W^1,\xi^1)$ and $\mathscr L^2(X^2_{\cdot \wedge T},\mu^2,B^2,W^2,\xi^2)$ are equal. 
\end{lemma}

\begin{proof}
	Given an arbitrary solution $(X,\mu,B,W,\xi)$ to \eqref{eq rsdelem} on a probability space $(\Omega,\cF,\P)$, with $\nu:=\mathscr L(\mu,B,W,\xi)$, define an equivalent probability measure $\bQ_T$ by $$\frac{d\bQ_T}{d\P}:=\mathcal E_T\bigg (-\int_0^\cdot\sigma^{-1}(s,X)b(s,X,\mu)dW_s\bigg ).$$
	As $(\mu,B,\xi)\indep W$, the tuple $(\mu,B,\xi)$ has the same joint distribution under $\bQ_T$ or $\P$. By Girsanov's Theorem, $\tilde W:=W+\int_0^{\cdot\wedge T} \sigma^{-1}(s,X)b(s,X,\mu)ds$ is a $\bQ_T$-Brownian motion. Therefore, $(\mu,B,\tilde W,\xi)\sim\nu$ under $\bQ_T$. Also, since $X$ satisfies \eqref{eq driftless} on $[0,T]$ under $\bQ_T$, with stochastic input $(B,\tilde W,\xi)$, the process $X_{\cdot\wedge T}$ has a uniquely determined law on $\bQ_T$ since $\eqref{eq driftless}$ has a unique strong solution.
	
	Combining these facts, under $\bQ_T$, $ (X_{\cdot\wedge T},\mu,B,\tilde W,\xi)$ has a joint distribution that does not depend upon the choice of weak solution. This uniquely determines their joint law with $W$ and $\mathcal E_T(\int_0^\cdot\sigma^{-1}(s,Y)b(s,Y,G(U,B))d\tilde W_s)$ under $\bQ_T$.
	
	Since $\P$ and $\bQ_T$ are equivalent, $$\P[(X_{\cdot\wedge T},\mu,B,W,\xi)\in A]=\E_{\bQ_T}\bigg [\frac{d\P}{d\bQ_T}\1_{ (X_{\cdot\wedge T},\mu,B,W,\xi)\in A}\bigg ].$$ 
	Further, since $\frac{d\P}{d\bQ_T}=\big (\frac{d\bQ_T}{d\P}\big )^{-1}$ one can write, 
	\begin{equation}
	\begin{split}
	\frac{d\P}{d\bQ_T} & = \exp\bigg\{\int_0^T\sigma^{-1}(s,X)b(s,X,\mu)dW_s+\frac{1}{2}\int_0^T|\sigma^{-1}(s,X)b(s,X,\mu)|^2ds\bigg\}\\
	& = \exp\bigg\{\int_0^T\sigma^{-1}(s,X)b(s,X,\mu)d\tilde W_s-\frac{1}{2}\int_0^T|\sigma^{-1}(s,X)b(s,X,\mu)|^2ds\bigg\}\\
	& =\mathcal E_T\bigg(\int_0^\cdot\sigma^{-1}(s,X)b(s,X,\mu)d\tilde W_s\bigg).
	\end{split}
	\end{equation}
	Finally, 
	\begin{equation}
	\notag
	\begin{split}
	 \P[(X_{\cdot\wedge T},\mu,B,  W,\xi)\in A]
	= & \E_{\bQ_T}\bigg[\frac{d\P}{d\bQ_T}\1_{ (X_{\cdot\wedge T},\mu,B,W,\xi)\in A}\bigg] \\
	= & \E_{\bQ_T}\bigg[\mathcal E_T\bigg(\int_0^\cdot\sigma^{-1}(s,X)b(s,X,\mu)d\tilde W_s\bigg)\1_{{ (X,\mu,B,\tilde W-\int_0^{\cdot\wedge T}\sigma^{-1}(s,X)b(s,X,\mu)ds,\xi)\in A}} \bigg],
	\end{split}
	\end{equation}	
	which does not depend upon the choice of weak solution. 
\end{proof}

\subsection{Proof of the Uniqueness Theorem}\label{sec uniqproof} To aid in the reading of this subsection, the strategy is briefly outlined as follows:
 \begingroup
 \rightskip\leftskip
 \paragraph{Proof Outline}
 \begin{itemize}
 	\leftskip2.6em
 	\item[Steps 1.-2.] Disintegrate the joint distributions of the solutions to identify the underlying randomness behind the flows of conditional distributions ($\mu^1$ and $\mu^2$).
 	\item[Steps 3.-4.] Introduce a Monge-Kantorovich Problem with a tailored cost function that forces the optimal coupling for this problem to constrain the underlying randomness to be the same for each solution. 
 	\item[Step 5.] Show that it is possible to represent the distributions of the solutions by a unique solution to the drift-less equation viewed on two probability spaces related by Girsanov transformations. This requires the uniqueness in law to a certain class of SDEs with random coefficients as given by Lemma \ref{lem uniqRSDE}.
 	\item[Step 6.] For a small time interval, estimate the distance between two processes' distributions by studying the \emph{dual} Kantorovich Problem, showing that for a small time interval, there is uniqueness in joint law.
 	\item[Step 7.] Conclude by induction.
 \end{itemize} 
 
 \endgroup
 \begin{proof}[Proof of Theorem \ref{thm jwkuniq}]
 	Given two weak solutions to \eqref{eq: MKVSDECNuniq} of the form given by Definition \eqref{def weaksoln},\newline $(X^1,\mu^1,B^1,W^1,\xi^1)$ and $(X^2,\mu^2,B^2,W^2,\xi^2)$, denote the laws of the solutions (with $\xi^i$ hidden inside $X^i$ since $\xi^i=X^i_0$) on their respective probability spaces by $$\mathscr L^1(X^1,\mu^1,B^1,W^1)\text{ and }\mathscr L^2(X^2,\mu^2,B^2,W^2),$$ where the superscript on $\mathscr L$ refers to the fact that these weak solutions may be defined on different probability spaces. In order to compare the distributions of the two solutions, one needs to couple the distributions on a probability space in such a way that fixes the underlying randomness of both $\mu^1$ and $\mu^2$ to be the same. This is done as follows:
 	\begin{enumerate}
 		\leftskip-1.2em
 		\item Disintegrate the joint distributions of the two solutions (see Chapter 10 in volume II of \cite{bogachevMeasureTheoryBook}) into the joint distributions of $(\mu^i,B^i,W^i)$ and the conditional distribution of $X^i$ \emph{given} $\mu^i,B^i,W^i$. This is written as 
 		$$\mathscr L^i(X^i,\mu^i,B^i,W^i)=p_X^i(dx,\mu,b,w)\mathscr L^i(\mu^i,B^i)(d\mu,db) \mathscr L^i(W^i)(dw),$$
 		using the independence of $W^i$ and $(\mu^i,B^i)$. 
 		\item From Blackwell and Dubins \cite{BlackwellDubins}, there exists for each $i\in\{1,2\},$ a measurable function $G^i:[0,1]\times C(I;\mathbb R^{d_B})\rightarrow C(I;\cP(\cC))$, such that, if on some probability space there are elements $U,B$ such that $U\sim \text{Unif}(0,1)=:\lambda$, $B\sim\mathscr L^i(B^i)$ and $U\indep B$, then $$\mathscr L(G^i(U,B),B)=\mathscr L^i(\mu^i,B^i).$$
 		Note that the functions $G^i$ cannot be claimed to be \emph{adapted} in the sense that, if for $b^1,b^2\in C(I;\mathbb R^{d_B})$ such that $b^1_{\cdot\wedge t}=b^2_{\cdot\wedge t}$ for some $t\in I$, then $G^i(u,b^1)_t=G^i(u,b^2)_t$. This is shown in Example 5.3 of \cite{lackerdensesets}.
 		
 		Letting $\mathcal W_d$ denote Wiener measure on $C(I;\mathbb R^d)$, consider for $i\in\{1,2\}$, $$\pi^i:=p^i_X(dx,\mu,b,w)\delta_{G^i(u,b)}(d\mu) \lambda(du)\mathcal W_{d_B}(db)\mathcal W_{d_W}(dw).$$ 
 		Equipping the space $E:=(\cC \times C(I;\cP(\cC)) \times [0,1] \times C(I;\mathbb R^{d_B}) \times C(I;\mathbb R^{d_W}))$ and its product $\sigma$-algebra with the measure $\pi^i$, the canonical random elements $(X,\mu,U,B,W)$ are such that $(X,\mu,B,W)$ have distribution $\mathscr L^i(X^i,\mu^i,B^i,W^i)$.
 		
 		Further, for $i\in\{1,2\}$, introduce the measure $$\pi^i_X:=p^i_X(dx,G^i(u,b),b,w)  \lambda(du)\mathcal W_{d_B}(db)\mathcal W_{d_W}(dw).$$ One can equip the product space $E^*:=(\cC \times [0,1] \times C(I;\mathbb R^{d_B}) \times C(I;\mathbb R^{d_W}))$ (with product $\sigma$-algebra denoted $\cB(E^*)$) with $\pi^i_X$ and define $\mu:=G^i(U,B)$. Then, the canonical random elements $X,U,B,W$ along with $\mu$ satisfy again, $\mathscr L^{\pi^i_X}(X,\mu,B,W)= \mathscr L^i(X^i,\mu^i,B^i,W^i)$ and consequently, denoting $(\Omega,\cF,\P):=(E^*,\cB(E^*),\pi^i_X)$, for any $A\in \cB(\cC)$ and bounded measurable $f:C(I;\cP(\cC))\times C(I;\mathbb R^{d_B})\rightarrow \bR$,
 		\begin{equation}\label{eq Gcnd}
 		\begin{split}
 		\E[G^i(U,B)_t(A)f(G^i(U,B),B)]= & \E^i[\mu^i_t(A)f(\mu^i,B^i)]\\
 		= & \E^i[\1_{A}(X^i_{\cdot\wedge t})f(\mu^i,B^i)]\\
 		= & \E[\1_A(X_{\cdot\wedge t})f(G^i(U,B),B)].\\
 		\end{split}
 		\end{equation}	
 		Hence, $\mu_t=G^i(U,B)_t=\mathscr L(X_{\cdot \wedge t}|G^i(U,B),B)=\mathscr L(X_{\cdot\wedge t}|\mu,B)$ for all $t\in I$. An important observation is that, since $X$ is independent of $U$ given $\sigma(G^i(U,B),B)$, $\mu_t=\mathscr L(X_{\cdot\wedge t}|U,B)$ for all $t\in I$.
 		\item On the product space $E^*\times E^*$, define the lower semi-continuous cost function 
 		\begin{equation}
 		c^*((x^1,u^1,b^1,w^1),(x^2,u^2,b^2,w^2)):= 
 		\begin{cases}
 		\1_{x^1\neq x^2}+d(w^1,w^2)\wedge 1\\
 		\quad\quad\,\,\text{ if }(u^1,b^1)=(u^2,b^2),\\
 		\infty  \quad\,\text{ otherwise.}\\
 		\end{cases}
 		%
 		\end{equation}
 		where $d$ is the uniform metric on $C(I;\mathbb R^{d_W})$. Let $W^*$ be the Monge-Kantorovich Problem (see Chapters 4 and 5 in \cite{villani2009}) with cost function $c^*$:
 		\begin{equation}\label{eq W^*}
 		\begin{split}
 		W^*(\pi^1_X,\pi^2_X):=\inf_{\pi: \,\pi \text{ couples }\pi^1_X,\pi^2_X}\int_{E^*\times E^*}c^*d\pi. \\
 		\end{split}
 		\end{equation}
 		There exists an optimal coupling for this problem (a coupling minimizing the expected cost $\int c^*d\pi$) since the $c^*$ is lower semi-continuous, see \cite{villani2009}, Theorem 4.1.
 		If $W^*(\pi^1_X,\pi^2_X)=0$, then one can conclude $\pi^1_X=\pi^2_X$ since the cost function $c^*((x^1,u^1,b^1,w^1),(x^2,u^2,b^2,w^2))=0$ if and only if $(x^1,u^1,b^1,w^1)=(x^2,u^2,b^2,w^2)$. Further, on the optimal coupling from \eqref{eq W^*}, following the argument  behind equation \eqref{eq Gcnd}, $$G^1(U,B)_t=\mathscr L(X^1_{\cdot\wedge t}|U,B)=\mathscr L(X^2_{\cdot\wedge t}|U,B)=G^2(U,B)_t,$$ almost surely for all $t\in I$, which by the continuity of sample paths of $G^i(U,B)$ is enough to claim that $G^1(U,B)$ and $G^2(U,B)$ are almost surely equal. It will consequently be the aim to show $W^*(\pi^1_X,\pi^2_X)=0$ for any two solutions to \eqref{eq: MKVSDECNuniq}. 
 		
 		First, note that by the gluing lemma there exists a probability space $(\tilde \Omega,\tilde \cF,\tilde \P)$ on which there are random elements $\tilde X^1,\tilde X^2,\tilde U,\tilde B,\tilde W^1,\tilde W^2$ with $\tilde{\mathscr L}(\tilde X^i,\tilde U,\tilde B,\tilde W^i)=\pi^i_X$.
 		It is easy to see that
 		\begin{equation}\notag
 		\begin{split}
 		W^*(\pi^1_X,\pi^2_X)\leq &  \tilde \E[c^*((\tilde X^1,\tilde U,\tilde B,\tilde W^1),(\tilde X^2,\tilde U,\tilde B,\tilde W^2))]\\
 		= & \tilde \E[\1_{\tilde X^1\neq \tilde X^2}+d(\tilde W^1,\tilde W^2)\wedge 1] \\
 		\leq & 2. \\
 		\end{split}
 		\end{equation}
 		On the other hand, for any coupling of $\pi^1$ and $\pi^2$ such that $\P[(U^1,B^1)\neq (U^2,B^2)]>0$, the quantity $\E[c^*((X^1, U^1,B^1,W^1),(X^2, U^2,B^2,W^2))]=\infty $. Therefore, the infimum (that is attained by some optimal coupling) in $W^*$ may be taken over all couplings ensuring $\P[(U^1,B^1)\neq (U^2,B^2)]=0$. By completing the probability space, it can be assumed that for the optimal coupling, $(U^1,B^1)= (U^2,B^2)$ surely and the superscripts will consequently be dropped. 
 	\end{enumerate}
 	To show that $W^*(\pi^1_X,\pi^2_X)=0$, it will first be shown that $W^*=0$ for solutions restricted to a short time interval. Define $p^i_{X,T}$ as the image of $p^i_X$ through the map $\cC\ni x\mapsto x_{\cdot\wedge T}\in \cC$. Then, defining 
 	$$\pi^i_{X,T}:=p^i_{X,T}(dx,G^i(u,b),b,w) \lambda(du)\mathcal W_{d_B}(db)\mathcal W_{d_B}(dw),$$
 	see that for $E^*$ equipped with $\pi^i_{X,T}$, and again defining $\mu^i:=G^i(U,B)$, the elements $X,\mu,B,W$ have distribution $\mathscr L^i(X^i_{\cdot\wedge T},\mu^i,B^i,W^i)$. It will be shown that for some small $T$, $W^*_T:=W^*(\pi^1_{X,T},\pi^2_{X,T})=0$ by representing the two measures via Girsanov transformations from the optimal coupling for $W^*_T$. Then, by repeating the argument, $W^*(\pi^1_X,\pi^2_X)=0$ will be established by induction on intervals $[0,kT]$. The optimal coupling for $W^*_T$, denoted $\P$ henceforth, satisfies $X^i=X^i_{\cdot\wedge T}$ and for all $t\leq T$,
 	\begin{equation}\label{eq muWest}
 	\begin{split}
 	\E[d_{TV}(\mu^1_t, \mu^2_t)] \leq   \E[ \E[\1_{ X^1_{\cdot\wedge t}\neq  X^2_{\cdot\wedge t}}|\cF^{B, U}]]=  \E[  \1_{ X^1_{\cdot\wedge t}\neq  X^2_{\cdot\wedge t}} ]\leq &  \E[  \1_{ X^1_{\cdot\wedge T}\neq  X^2_{\cdot\wedge T}} ]\\
 	& = W^*(\pi_{X,T}^1,\pi_{X,T}^2).\\
 	\end{split}
 	\end{equation}
 	The following argument shows that for small $T$, $W^*_T=0$:
 	\begin{enumerate}\addtocounter{enumi}{3}
 		\leftskip-1.2em
 		\item By the Kantorovich Duality (see Theorem 5.10 in \cite{villani2009}), the primal and dual Kantorovich problems for $c^*$ satisfy,
 		\begin{equation}\label{eq wasscomp1}
 		\begin{split}
 		W^*& (\pi_{X,T}^1,\pi_{X,T}^2)\\
 		= & \sup_{h\text{ $c^*$-convex}}\bigg ( \int h(x,u,b,w)(\pi^1_{X,T}- \pi^2_{X,T})(dx,du,db,dw)\bigg )\\
 		= & \sup_{h\text{ $c^*$-convex}} \E [ h(X^1,U,B,W^1)-h(X^2,U,B,W^2)] \\
 		\end{split}
 		\end{equation}
 		The second equality holds since $\P$ is a coupling of $\pi^1_{X,T}$ and $\pi^2_{X,T}$. 
 		The definition of $c^*$ convexity, can be found in \cite{villani2009} p.54, but for the purposes here it will suffice to consider the equivalence that, since $c^*$ satisfies the triangle inequality, $h$ is $c^*$-convex iff 
 		\begin{equation}
 		\label{eq equivconvex}
 		h(x^1,u^1,b^1,w^1)-h(x^2,u^2,b^2,w^2)\leq c^*((x^1,u^1,b^1,w^1),(x^2,u^2,b^2,w^2)).
 		\end{equation}
 		It will be necessary to consider an alternative, but equivalent supremum in the right hand side of Equation \eqref{eq wasscomp1}, where one is able to assume that all functions $h$ in the supremum are non-negative and bounded. This will be arrived at by the subsequent argument. 
 		
 		By the characterisation of $c^*$-convex functions, \eqref{eq equivconvex}, for arbitrary but fixed $x'\in\cC$ and $w'\in C(I;\bR^{d_W})$, mapping every $c^*$-convex function $h$ to a new $c^*$-convex function $h'$ such that $$h'(x,u,b,w):=h(x,u,b,w)-h(x',u,b,w')\leq c^*((x,u,b,w),(x',u,b,w')),$$ one can see that since $c^*$ is symmetric, $|h'|\leq 2$. Finally, setting $h'':=h'+2$ (again $h''$ is $c^*$-convex), see that for every $c^*$-convex $h$, 
 		\begin{equation}\notag
 		\begin{split}
 		& \E [ h(X^1,U,B,W^1)-h(X^2,U,B,W^2)]\\
 		&=\E [ h''(X^1,U,B,W^1)-h''(X^2,U,B,W^2)]
 		\end{split}
 		\end{equation}
 		and $h''$ is $[0,4]$ valued. Therefore, by sending every $h$ to its corresponding $h''$,
 		\begin{equation}\label{eq wasscomp}
 		\begin{split}
 		\hspace{-8mm}W^*& (\pi_{X,T}^1,\pi_{X,T}^2) =  \sup_{h:E^*\to[0,4],\text{ $c^*$-convex }} \E [ h(X^1,U,B,W^1)-h(X^2,U,B,W^2)]. \\ 
 		\end{split}
 		\end{equation}
 		\item Now, on the optimal probability space $(\Omega,\cF,\P)$, enlarged to include another Brownian motion $W^0$ (this is not necessary, since one could use $W^1$ or $W^2$ in place of $W^0$, but arguably this eases notation), there is a strong solution $X^0$ to the driftless equation \eqref{eq driftless} by Assumption \ref{ass uniq}. Indeed, there is a process $X^0$ such that 
 		\begin{equation}\notag
 		dX^0_t=\sigma(t,X^0)dW^0_t+\rho(t,X^0)dB_t.
 		\end{equation}
 		In order to estimate the right hand side of \eqref{eq wasscomp}, it is critical to represent the distributions of $X^i_{\cdot\wedge T}$ by the distributions of $X^0_{\cdot\wedge T}$ under suitable Girsanov transformations. For each $i=1,2$, define measures $\bQ^{i}\sim \P$ by 
 		\begin{equation}\label{eq girsanov}
 		\begin{split}
 		\frac{d\bQ^i}{d\P}&
 		:= \mathcal E\bigg ( \int_0^{\cdot\wedge T}\sigma^{-1}(s,X^0)b(s,X^0,\mu^i)dW^0_s\bigg )_{\infty}.
 		\end{split}
 		\end{equation}
 		$\mathcal E(M)_t$ denotes the Dol\'eans-Dade exponential of $M$ at time $t$, $\mathcal E(M)_t:=\exp \{M_t-\frac{1}{2}[M]_t\}$. These changes of probability measure are well defined due to the assumption of boundedness of $\sigma^{-1}b$. By Girsanov's Theorem, $W^{0,i}:=W^0-\int_0^{\cdot\wedge T}\sigma^{-1}(s,X^0)b(s,X^0,\mu^i)ds$ is a $\bQ^i$ Brownian motion on $I$, and on $[0,T]$ and for each $i=1,2$,
 		$$dX^0_t=b(t,X^0,\mu^i)dt+\sigma(t,X^0)dW^{0.i}_t+\rho(t,X^0)dB_t.$$  
 		It is now claimed that, $\mathscr L^i(X^0_{\cdot\wedge T},U,B,W^{0,i})=\mathscr L(X^i_{\cdot\wedge T},U,B,W^i)$, where $\mathscr L^i$ denotes the law on $\bQ^i$ (and continues to do so for the remainder of the proof). 
 		This follows from the uniqueness in joint law on $[0,T]$ for solutions for SDEs with random coefficients of the form:
 		\begin{equation}\label{eq rsde}
 		dY_t=b(t,Y,\mu)dt+\sigma(t,Y)dW_t+\rho(t,Y)dB_t,
 		\end{equation}
 		where the joint distribution of $(\mu,B,W)$ is determined. This uniqueness is given by Lemma \ref{lem uniqRSDE}, which is stated and proved at the end of the current proof.

 		\item Recalling the equation \eqref{eq wasscomp}, and the two equivalent probability spaces $\bQ^1$ and $\bQ^2$,
 		\begin{equation}\label{eq wasscompRN}
 		\hspace{-9mm}\begin{split}
 		& W^*(\pi_{X,T}^1,\pi_{X,T}^2)\\
 		= & \sup_{h:E^*\to[0,4],\text{ c-convex }} \E [ h(X^1,U,B,W^1)-h(X^2,U,B,W^2)]\\
 		= &\sup_{h:E^*\to[0,4],\text{ c-convex }} \E^1 [h(X^0_{\cdot\wedge T},U,B,W^{0,1})]-\E^2[h(X^0_{\cdot\wedge T},U,B,W^{0,2})] \\
 		= & \sup_{h:E^*\to[0,4],\text{ c-convex }} \E\bigg[\frac{d\bQ^1}{d\P}h(X^0_{\cdot\wedge T},U,B,W^{0,1})-\frac{d\bQ^2}{d\P}h(X^0_{\cdot\wedge T},U,B,W^{0,2})\bigg] \\
 		= & \sup_{h:E^*\to[0,4],\text{ c-convex }} \bigg \{ \E\bigg[\frac{d\bQ^1}{d\P}\bigg(h(X^0_{\cdot\wedge T},U,B,W^{0,1})-h(X^0_{\cdot\wedge T},U,B,W^{0,2})\bigg)\bigg]\\
 		& \quad\quad\quad\quad\quad\quad\quad\quad+ \E\bigg[\bigg (\frac{d\bQ^1}{d\P}-\frac{d\bQ^2}{d\P}\bigg )h(X^0_{\cdot\wedge T},U,B,W^{0,2})\bigg]\bigg \} \\
 		\end{split}
 		\end{equation} 
 		The right hand side of \eqref{eq wasscompRN} will be estimated as follows:
 		\begin{equation}\notag
 		\hspace{-9mm}\begin{split}
 		& \sup_{h:E^*\to[0,4],\text{ c-convex }} \bigg \{ \E\bigg[\frac{d\bQ^1}{d\P}\bigg(h(X^0_{\cdot\wedge T},U,B,W^{0,1})-h(X^0_{\cdot\wedge T},U,B,W^{0,2})\bigg)\bigg]\\
 		& \quad\quad\quad\quad\quad\quad\quad\quad +  \E\bigg[\bigg(\frac{d\bQ^1}{d\P}-\frac{d\bQ^2}{d\P}\bigg)h(X^0_{\cdot\wedge T},U,B,W^{0,2})\bigg]\bigg \} \\
 		\leq & \sup_{h:E^*\to[0,4],\text{ c-convex }} \E^1[ (h(X^0_{\cdot\wedge T},U,B,W^{0,1})-h(X^0_{\cdot\wedge T},U,B,W^{0,2}))]\\
 		& +\sup_{h:E^*\to [0,4],\text { measurable}} \E^1\bigg[\bigg(1-\frac{d\bQ^2}{d\P^1}\bigg)h(X^0_{\cdot\wedge T},U,B,W^{0,2})\bigg] \\
 		 		\end{split}
 		\end{equation}
 		\begin{equation}\label{eq expocomp}
 		\begin{split}
 		\leq & \E^1[ d(W^{0,1},W^{0,2})\wedge 1]  +4 \E^1\bigg[\bigg(1-\frac{d\bQ^2}{d\bQ^1}\bigg)  \1_{\frac{d\bQ^2}{d\bQ^1}<1}   \bigg] \\
 		\leq & \E^1[ d(W^{0,1},W^{0,2}) ]  +4 \E^1\bigg[\bigg|1-\frac{d\bQ^2}{d\bQ^1}\bigg|  \1_{\frac{d\bQ^2}{d\bQ^1}<1}   \bigg] \\
 		%
 		\end{split}
 		\end{equation}
 		Recalling the definitions of $W^i$ and the form of $\frac{d\bQ^1}{d\P}$ and $\frac{d\bQ^2}{d\P}$ from \eqref{eq girsanov}, $\frac{d\bQ^2}{d\bQ^1}$ can be rewritten as follows:
 		\begin{equation}
 		\hspace{-9mm}\begin{split}
 		\frac{d\bQ^2}{d\bQ^1}  = & \exp\bigg\{ \int_0^T\sigma^{-1}(s,X^0)b(s,X^0,\mu^2)dW^{0}_s- \int_0^T\sigma^{-1}(s,X^0)b(s,X^0,\mu^1)dW^{0}_s\\
 		& + \frac{1}{2}\int_0^T|\sigma^{-1}(s,X^0)b(s,X^0,\mu^1)|^2 ds -\frac{1}{2}\int_0^T|\sigma^{-1}(s,X^0) b(s,X^0,\mu^2)|^2 ds \bigg\}   \\
 		= &  \exp\bigg\{ \int_0^T\sigma^{-1}(s,X^0)b(s,X^0,\mu^2)dW^{0,1}_s- \int_0^T\sigma^{-1}(s,X^0)b(s,X^0,\mu^1)dW^{0,1}_s\\
 		& - \frac{1}{2}\int_0^T|\sigma^{-1}(s,X^0)b(s,X^0,\mu^1)|^2 ds -\frac{1}{2}\int_0^T|\sigma^{-1}(s,X^0) b(s,X^0,\mu^2)|^2 ds  \\
 		& +\int_0^T\sigma^{-1}(s,X^0)b(s,X^0,\mu^1)\cdot \sigma^{-1}(s,X^0)b(s,X^0,\mu^2)ds \bigg\} \\
 		= & \exp\bigg\{ -\int_0^T\sigma^{-1}(s,X^0)b(s,X^0,\mu^1)- \sigma^{-1}(s,X^0)b(s,X^0,\mu^2)dW^{0,1}_s\\
 		& - \frac{1}{2} \int_0^T|\sigma^{-1}(s,X^0)b(s,X^0,\mu^1)- \sigma^{-1}(s,X^0)b(s,X^0,\mu^2)|^2ds \bigg\}.\\
 		\end{split}
 		\end{equation} 
 		Now, on the event $\frac{d\bQ^2}{d\bQ^1}<1$, 
 		\begin{equation}\notag
 		\begin{split}
 		& \exp\bigg\{ -\int_0^T\sigma^{-1}(s,X^0)b(s,X^0,\mu^1)- \sigma^{-1}(s,X^0)b(s,X^0,\mu^2)dW^{0,1}_s\\
 		& - \frac{1}{2} \int_0^T|\sigma^{-1}(s,X^0)b(s,X^0,\mu^1)- \sigma^{-1}(s,X^0)b(s,X^0,\mu^2)|^2ds \bigg\} <1 .\\
 		\end{split}
 		\end{equation}
 		Since for all $x\leq 0$ (i.e. $e^x<1$), $|1-e^x|\leq |x|$,
 		\begin{equation}\notag
 		\begin{split}
 		 \E^1\bigg[\bigg|1-\frac{d\bQ^2}{d\bQ^1}  \bigg|\1_{\frac{d\bQ^2} {d\bQ^1}<1}\bigg] \leq & \E^1\bigg[\bigg|   -\int_0^T\sigma^{-1}(s,X^0)b(s,X^0,\mu^1)- \sigma^{-1}(s,X^0)b(s,X^0,\mu^2)dW^{0,1}_s\\
 		& - \frac{1}{2} \int_0^T|\sigma^{-1}(s,X^0)b(s,X^0,\mu^1)- \sigma^{-1}(s,X^0)b(s,X^0,\mu^2)|^2ds  \bigg|\1_{\frac{d\bQ^2} {d\bQ^1}<1}\bigg]\\
 		\leq & \E^1\bigg[\bigg|   \int_0^T\sigma^{-1}(s,X^0)b(s,X^0,\mu^1)- \sigma^{-1}(s,X^0)b(s,X^0,\mu^2)dW^{0,1}_s\bigg|\\
 		& +  \frac{1}{2} \int_0^T|\sigma^{-1}(s,X^0)b(s,X^0,\mu^1)- \sigma^{-1}(s,X^0)b(s,X^0,\mu^2)|^2ds  \bigg]\\
 		\leq & \E^1\bigg[\sup_{t\leq T}\bigg|   \int_0^t\sigma^{-1}(s,X^0)b(s,X^0,\mu^1)- \sigma^{-1}(s,X^0)b(s,X^0,\mu^2)dW^{0,1}_s\bigg|\bigg]\\
 		& +  \frac{1}{2}  \E^1\bigg[ \int_0^T|\sigma^{-1}(s,X^0)b(s,X^0,\mu^1)- \sigma^{-1}(s,X^0)b(s,X^0,\mu^2)|^2ds  \bigg].\\
 		\end{split}
 		\end{equation} 
 		Applying the Burkh\"older-Davis-Gundy inequality (the corresponding constant denoted $c_{\text{\tiny BDG}}$),
 		\begin{equation}\notag
 		\hspace{-8mm}\begin{split}
 		& \E^1\bigg[\bigg|1-\frac{d\bQ^2}{d\bQ^1}  \bigg|\1_{\frac{d\bQ^2} {d\bQ^1}<1}\bigg] \\
 		\leq & c_{\text{\tiny BDG}} \E^1\bigg[   \bigg(\int_0^T|\sigma^{-1}(s,X^0)b(s,X^0,\mu^1)- \sigma^{-1}(s,X^0)b(s,X^0,\mu^2)|^2ds\bigg)^\frac{1}{2}\bigg]\\
 		& +  \frac{1}{2} \E^1\bigg[  \int_0^T|\sigma^{-1}(s,X^0)b(s,X^0,\mu^1)- \sigma^{-1}(s,X^0)b(s,X^0,\mu^2)|^2ds  \bigg]\\
 		\end{split}
 		\end{equation} 
 		Now, using the assumption of total variation Lipschitz continuity of $\sigma^{-1}b$ in the measure component,
 		\begin{equation}
 		\hspace{-8mm}\begin{split}
 		& c_{\text{\tiny BDG}} \E^1\bigg[   \bigg(\int_0^T|\sigma^{-1}(s,X^0)b(s,X^0,\mu^1)- \sigma^{-1}(s,X^0)b(s,X^0,\mu^2)|^2ds\bigg)^\frac{1}{2}\bigg]\\
 		& +  \frac{1}{2} \E^1\bigg[  \int_0^T|\sigma^{-1}(s,X^0)b(s,X^0,\mu^1)- \sigma^{-1}(s,X^0)b(s,X^0,\mu^2)|^2ds  \bigg]\\
 		\leq   & c_{\text{\tiny BDG}} c_{\text{\tiny TV}} \E^1\bigg[   \bigg(\int_0^T d_{\text{\tiny TV}}(\mu^1_s,\mu^2_s)^2ds\bigg)^\frac{1}{2}\bigg]+  \frac{1}{2}c_{\text{\tiny TV}}^2  \E^1\bigg[  \int_0^Td_{\text{\tiny TV}}(\mu^1_s,\mu^2_s)^2 ds  \bigg].\\
 		\end{split}
 		\end{equation}
 		And since for all $s\leq T$, $d_{\text{\tiny TV}}(\mu^1_s,\mu^2_s)\leq d_{\text{\tiny TV}}(\mu^1_T,\mu^2_T)$, 
 		\begin{equation}\notag
 		\hspace{-10mm}\begin{split}
 		& \E^1\bigg[\bigg|1-\frac{d\P^2}{d\P^1} \bigg|\1_{\frac{d\P^2} {d\P^1}<1}\bigg] \\
 		\leq   & c_{\text{\tiny BDG}} c_{\text{\tiny TV}} T^\frac{1}{2} \E^1[      d_{\text{\tiny TV}}(\mu^1_T,\mu^2_T) ]+  \frac{1}{2}c_{\text{\tiny TV}}^2T  \E^1[  d_{\text{\tiny TV}}(\mu^1_T,\mu^2_T)^2  ]\\
 		=   & c_{\text{\tiny BDG}} c_{\text{\tiny TV}} T^\frac{1}{2} \E[      d_{\text{\tiny TV}}(\mu^1_T,\mu^2_T) ]+  \frac{1}{2}c_{\text{\tiny TV}}^2T  \E[  d_{\text{\tiny TV}}(\mu^1_T,\mu^2_T)^2  ]\\
 		\leq    & c_{\text{\tiny BDG}} c_{\text{\tiny TV}} T^\frac{1}{2} \E[  \E[\1_{X^1_{\cdot \wedge T}\neq X^2_{\cdot \wedge T}}|U,B]  ]+  \frac{1}{2}c_{\text{\tiny TV}}^2T  \E[  \E[\1_{X^1_{\cdot \wedge T}\neq X^2_{\cdot \wedge T}}|U,B]^2  ]\\
 		\leq    & ( c_{\text{\tiny BDG}} c_{\text{\tiny TV}} T^\frac{1}{2}  +  \frac{1}{2}c_{\text{\tiny TV}}^2T ) \E[  \E[\1_{X^1_{\cdot \wedge T}\neq X^2_{\cdot \wedge T}}|U,B]   ]\\
 		=  & ( c_{\text{\tiny BDG}} c_{\text{\tiny TV}} T^\frac{1}{2}  +  \frac{1}{2}c_{\text{\tiny TV}}^2T ) \P[ {X^1_{\cdot \wedge T}\neq X^2_{\cdot \wedge T}}   ]\\
 		=  & ( c_{\text{\tiny BDG}} c_{\text{\tiny TV}} T^\frac{1}{2}  +  \frac{1}{2}c_{\text{\tiny TV}}^2T ) W^*(\pi^1_{X,T},\pi^2_{X,T}).\\
 		\end{split}
 		\end{equation}
 		Similarly, for $\E^1[ d(W^{0,1},W^{0,2}) ] $, one estimates
 		\begin{equation}
 		\notag
 		\hspace{-10mm}\begin{split}
 		& \E^1[ d(W^{0,1},W^{0,2}) ]\\
 		 \leq &  \E^1\bigg[\sup_{t\leq T}\bigg|\int_0^t\sigma^{-1}(s,X^0)b(s,X^0,\mu^1)- \sigma^{-1}(s,X^0)b(s,X^0,\mu^2)ds\bigg|\bigg]\\
 		\leq & c_{\text{\tiny TV}}T\E^1[ d_{\text{\tiny TV}}(\mu^1_T,\mu^2_T) ]\\
 		\leq & c_{\text{\tiny TV}}TW^*(\pi^1_{X,T},\pi^2_{X,T}).\\
 		\end{split}
 		\end{equation}
 		Putting the above two estimates together with \eqref{eq expocomp},
 		$$\hspace{-8mm}W^*(\pi^1_{X,T},\pi^2_{X,T})\leq (c_{\text{\tiny TV}}T+4 ( c_{\text{\tiny BDG}} c_{\text{\tiny TV}} T^\frac{1}{2}  +  \frac{1}{2}c_{\text{\tiny TV}}^2T ))W^*(\pi^1_{X,T},\pi^2_{X,T}).$$
 		Hence, choosing $T$ small enough such that $ c_{\text{\tiny TV}}T+4( c_{\text{\tiny BDG}} c_{\text{\tiny TV}} T^\frac{1}{2}  +  \frac{1}{2}c_{\text{\tiny TV}}^2T )=\alpha<1$, one has 
 		$$W^*(\pi^1_{X,T},\pi^2_{X,T})\leq \alpha W^*(\pi^1_{X,T},\pi^2_{X,T}).$$
 		This implies that $W^*(\pi^1_{X,T},\pi^2_{X,T})=0$. Importantly, this further implies that almost surely, $G^1(U,B)_{\cdot \wedge T}=G^2(U,B)_{\cdot \wedge T}$. Indeed, since $G^i(U,B)_t=\mu^i_t=\mathscr L(X^i_{\cdot\wedge t}|U,B)$, for any $t\leq T$, and any $A\in \cB(\cC)$,
 		\begin{equation}\notag
 		\begin{split}
 		\E[\mu^1_t(A)f(U,B)]=\E[\1_A(X^1_{\cdot\wedge t})f(U,B)]=& \E[\1_A(X^2_{\cdot\wedge t})f(U,B)]\\
 		=&  \E[\mu^2_{\cdot\wedge t}(A)f(U,B)].
 		\end{split}
 		\end{equation}
 		This means that the distribution of $(G^1(U,B)_{\cdot\wedge T},G^2(U,B)_{\cdot\wedge T})$ is concentrated on the diagonal (and will be on any probability space supporting $(U,B)$ with the same distribution). 
 		\item The result of the proof will follow by an inductive argument. Assume that for some $k\in\bN$ $W^*(\pi^1_{X,kT},\pi^2_{X,kT})=0$, then repeating the above argument for $\pi^1_{X,(k+1)T}$ and $\pi^2_{X,(k+1)T}$, then, since $\mu^1=\mu^2$ almost surely on $[0,kT]$,
 		\begin{equation}\notag
 		\hspace{8mm}\begin{split}
 		& W^*(\pi_{X,(k+1)T}^1,\pi_{X,(k+1)T}^2)  \\
  		\leq   & 4c_{\text{\tiny BDG}} c_{\text{\tiny TV}} \E^1\bigg[ \bigg(\int_0^{(k+1)T} d_{\text{\tiny TV}}(\mu^1_s,\mu^2_s)^2ds\bigg)^\frac{1}{2}\bigg] + 4 \frac{1}{2}c_{\text{\tiny TV}}^2  \E^1\bigg[  \int_0^{(k+1)T} d_{\text{\tiny TV}}(\mu^1_s,\mu^2_s)^2 ds  \bigg]\\
 		& + c_{\text{\tiny TV}}\E^1\bigg[  \int_0^{(k+1)T} d_{\text{\tiny TV}}(\mu^1_s,\mu^2_s) ds  \bigg]\\
 		= &  4c_{\text{\tiny BDG}} c_{\text{\tiny TV}} \E^1\bigg[   (\int_{kT}^{(k+1)T} d_{\text{\tiny TV}}(\mu^1_s,\mu^2_s)^2ds)^\frac{1}{2}\bigg] + 4 \frac{1}{2}c_{\text{\tiny TV}}^2  \E^1\bigg[  \int_{kT}^{(k+1)T} d_{\text{\tiny TV}}(\mu^1_s,\mu^2_s)^2 ds  \bigg]\\
 		& + c_{\text{\tiny TV}}\E^1\bigg[  \int_{kT}^{(k+1)T} d_{\text{\tiny TV}}(\mu^1_s,\mu^2_s) ds  \bigg]\\
 		\leq  &  (c_{\text{\tiny TV}}T+4 ( c_{\text{\tiny BDG}} c_{\text{\tiny TV}} T^\frac{1}{2}  +  \frac{1}{2}c_{\text{\tiny TV}}^2T )) W^*(\pi^1_{X,(k+1)T},\pi^2_{X,(k+1)T}).\\
 		\end{split}
 		\end{equation} 
 		Therefore $W^*(\pi^1_{X,(k+1)T},\pi^2_{X,(k+1)T})=0$. By induction, the proof is complete. 
 	\end{enumerate}
 \end{proof}

\section*{Acknowledgements}
We would like to express our gratitude to Sandy Davie from the University of Edinburgh and Daniel Lacker from Columbia University for discussions regarding this work and their helpful suggestions.


\appendix
\section*{Appendix}
\setcounter{section}{1}
\setcounter{theorem}{0} 
\setcounter{equation}{0}

The following lemma is standard and numerous lemmas of this type are proved in the note \cite{Taraldsen2018DoobDynkin}.
\begin{lemma}[Doob-Dynkin Lemma]\label{lem doobdynkin}Given measurable spaces $(\Omega,\cF)$, $(\cX,\cF_\cX)$ and $(\cY,\cF_\cY)$, with measurable functions $X:\Omega\mapsto \cX$ and $Y:\Omega\mapsto \cY$, if the image $X(\Omega)$ of function $X$ is contained in a standard Borel space, and $X$ is measurable with respect to the initial $\sigma$-algebra of $Y$ (the initial sigma algebra of $Y$ is defined as $\sigma( Y^{-1}(A):A\in \cF_\cY)$), then there exists a measurable $\phi:\cY\mapsto \cX$ such that $X=\phi(Y)$.
\end{lemma} 

\subsection{Immersion and Compatibility}\label{sec immcomp}

The following theorem follows from \cite{lackerdensesets} where further equivalent conditions and references can be found. 
\begin{theorem}[Conditions equivalent to Immersion]\label{thm compat}
On a given probability space $(\Omega,\cF,\bP)$, consider two filtrations $\bF, \mathbb G$ such that $\bF\subset  \mathbb G$. Then $\mathbb F$ is immersed in $\mathbb G$ under $\P$ if and only if any of the following conditions holds:
\begin{enumerate}
	\item $\cG_t$ is conditionally independent of $\cF_{\infty}$ given $\cF_t$, for any t. 
	\item Every bounded $\bF$ martingale is a $\mathbb G$ martingale. 
	\item For every t and every integrable $\cF_\infty$ measurable $X$, $\E[X|\cF_t]=\E[X|\cG_t]$ $\P$-a.s. 
	\item For every t and every integrable $\cG_t$ measurable $X$, $\E[X|\cF_t]=\E[X|\cF_\infty]$ $\P$-a.s.
\end{enumerate}
\end{theorem}

\subsection{Kolmogorov Continuity and Tightness}
The following two theorems are taken from \cite{KallenbergBook2002} on pages 57 and 313 respectively, where they are proved in sufficient generality for the present purposes. The statements have been adjusted, but remain true.
\begin{theorem}[Kolmogorov Continuity]\label{thm kcont}
Let $X$ be a process on $I$ with values in a Polish space $(\cY,d_{\cY})$ and assume that for some constants $a,b,c>0$ and any $s,t\in I$ such that $|t-s|\leq 1$ 
$$\mathbb E [d_{\cY}(X_t-X_s)^a]\leq c |t-s|^{1+b}.$$
 Then, X has a continuous version and for any $\gamma\in (0,b/a)$ the latter is almost surely locally $\gamma$ H\"older continuous. 
\end{theorem}
\begin{theorem}\label{thm kcontn}
Let $\{X^n\}$ be a family of continuous processes on $I$ with values in a Polish space $(\cY,d_{\cY})$. Assume that $\{X^n_0\}$ is tight and that for some constants $a,b,c>0$ and any $s,t\in I$ such that $|t-s|\leq 1$ and uniformly in $n\in \bN$,
$$\mathbb E [d_{\cY}(X^n_t-X^n_s)^a]\leq c |t-s|^{1+b}.$$
Then, $\{X^n\}$ is tight in $C(I,\cY)$  and for any $\gamma\in (0,b/a)$ the limiting processes are almost surely locally $\gamma$ H\"older continuous. 
\end{theorem}

\subsection{Lemmas \ref{thm stochfub} and \ref{lem cndp}}\label{sec proofs}

The authors expect that the following lemma has been proved elsewhere, but cannot yet find a reference. 

\begin{lemma}[Fubini-type Theorem for Conditional Expectation and \Ito\,Integrals]\label{thm stochfub}
	Given a probability space $(\Omega,\cF  ,\P)$, three filtrations $\bF^j:=(\cF^j_t)_{t\in I}$ $j=1,2,3$ and three processes $B,H,W$ satisfying the following conditions:
	\begin{enumerate}[i)]
		\item $\bF^1\subseteq \bF^2\subseteq\bF^3$ i.e. $\forall t\in I$, $\cF^1_t\subseteq \cF^2_t\subseteq \cF^3_t$.
		\item $\bF^1$ is immersed in $\bF^2$ under $\P$.
		\item $H$ is a bounded $\bF^2$-predictable process.
		\item $B$ and $W$ are $\bF^3$ Brownian Motions.
		\item $B$ is $\bF^1$ adapted.
		\item For any $s,t\in I$, $s\leq t$, $\sigma(W_r-W_s:s\leq r\leq t)\indep \bF^1_t\vee\bF^2_s$.
	\end{enumerate}
	Then the following hold $\P$-a.s. for all $t\in I$:
	\begin{equation}
	\label{eq stochfub1}
	\E \bigg [\int_0^tH_{s}\,dW_s\bigg |\cF^1_t\bigg ]=0,
	\end{equation}
	\begin{equation}
	\label{eq stochfub2}
	\E \bigg [\int_0^tH_s\,dB_s\bigg |\cF^1_t\bigg ]=\int_0^t\E  [H_s|\cF^1_s ]\,dB_s.
	\end{equation}
\end{lemma}	 

\begin{proof}[Proof of Lemma \ref{thm stochfub}]
	
	The proof will follow a monotone class argument. Firstly, equations \eqref{eq stochfub1} and \eqref{eq stochfub2} are shown to hold for the family of simple predictable processes. 
	
	Let $H^n$ be a simple predictable process defined by $$H^n_t:=Z^0\1_{\{0\}}(t)+\sum_{i=0}^{n-1}Z^i\1_{(t_i,t_{i+1}]}(t)$$ where $n\in \bN$, $0\leq t_0\leq \cdots\leq t_i\leq \cdots \leq t_n < \infty$ and $Z^i$ are bounded $\cF^2_{t_i}$ measurable random elements for all $i=0,..,n$.
	Then \eqref{eq stochfub1} is verified via the following:
	
	\begin{equation}\notag
	\begin{split}
	\E \bigg [\int_0^tH^n_{s}\,dW_s\bigg |\cF^1_t\bigg ]= & \sum_{i=0}^{n-1}\E [Z^i(W_{{t_{i+1}}\wedge t}-W_{t_i})|\cF^1_t]\\
	=& \sum_{i=0}^{n-1}\E[\E[ Z^i(W_{t_{i+1}\wedge t}-W_{t_i})|\cF^1_t\vee\cF^2_{t_i}]|\cF^1_t]\\
	=& \sum_{i=0}^{n-1}\E[\E[(W_{t_{i+1}\wedge t}-W_{t_i})|\cF^1_t\vee\cF^2_{t_i}]Z^i|\cF^1_t]\\
	= & 0. \\
	\end{split}
	\end{equation}
	The first equality follows from $H^n$ being a simple predictable process, the second and third from the tower and pull out properties of conditional expectation respectively, the fourth from condition $iv)$ and $vi)$.
	
	To verify the second equation \eqref{eq stochfub2}, consider the following equalities:
	
	\begin{equation}
	\notag
	\begin{split}
	\E\bigg [
	\int_0^t H^n_{s}\,dB_s\bigg |\cF^1_t\bigg ]= & \E\bigg [ \sum_{i=0}^{n-1}Z^i(B_{t_{i+1}\wedge t}-B_{t_{i}\wedge t})\bigg |\cF^1_t\bigg ]\\
	= &  \sum_{i=0}^{n-1}\E[Z^i|\cF^1_t](B_{t_{i+1}\wedge t}-B_{t_{i}\wedge t})\\
	= & \sum_{i=0}^{n-1}\E[Z^i|\cF^1_{t_i}](B_{t_{i+1}\wedge t}-B_{t_{i}\wedge t})\\
	= & \int_0^t\E[H^n_s|\cF^1_s]\,dB_s.\\
	\end{split}
	\end{equation}
	The second equality can be seen to hold by considering separately the cases: $t<t_i$, $t_i\leq t\leq t_{i+1}$ and $t_{i+1}<t$. The third equality holds from the immersion of $\bF^1$ in $\bF^2$ and the fourth from the definition of $H^n$. 
	
	Now that the desired equalities have been established for simple predictable processes, it remains to show the equality holds for a predictable process $H$ satisfying $iii)$ with a sequence of simple predictable processes $H^n\rightarrow H$ in uniformly on compact sets in probability (in ucp) as $n\rightarrow \infty$. Note that the sequence $H^n$ can be chosen such that for any $n\in\bN$, $|H^n|<K$, where $K$ is the bound for $H$. Recall that convergence in ucp means that for any $t\in I$, $\sup_{0\leq s \leq t}|H^n_s-H_s|$ converges to $0$ in probability. Hence there exists a subsequence $n_k$ that elevates the convergence to almost sure convergence along this subsequence. Therefore, by application of the dominated convergence for stochastic integrals [Theorem 32 p.145 \cite{ProtterBook1990}](with another subsequence) and dominated convergence for conditional expectation, the lemma is proved.
	
\end{proof}

%

\begin{lemma}\label{lem cndp}
	Given a probability space $(\Omega,\cF,\bP)$ supporting a continuous $\mathbb R^{d_X}$ valued stochastic process $X$ on the interval $I$. Suppose that for any $T<\infty $, $\mathbb E[\sup_{t\in I:t\leq T}|X_t|^p]<\infty$. Then for a filtration $\bF=(\cF_t)_{t\in I}$ there is a $\cP_p(\cC)$ valued $\bF$ adapted stochastic process $\mu$ such that for all $t\in I$, $\mu_t=\mathscr L(X_{\cdot\wedge t}|\cF_t)_{t\in I}$ i.e. $\mu_t$ is a regular conditional distribution of $X_{\cdot\wedge t}$ given $\cF_t$.
\end{lemma} 

\begin{proof}[Proof of Lemma \ref{lem cndp}]
	For each $t\in I$, use the existence theorem for regular conditional distributions to get hold of a stochastic kernel $\kappa_{X_{\cdot\wedge t},\cF_t}$, a $(\Omega,\cF_t)\rightarrow (\cP(\cC),\cB (\cP( \cC)))$ measurable function.
	
	Let $D_t:=\{\omega:\kappa_{X_{\cdot\wedge t},\cF_t}\notin \cP_p(\cC)\}$. To see that $D_t$ is in $ {\cF}_t$ first note that for some fixed $\eta \in \cP_p(\cC)$, the sets defined $A_\varepsilon^{\eta}:=\{\nu\in\cP_p(\cC):W_p(\nu,\eta)<\varepsilon \}$ for any $\varepsilon>0 $, are in $\mathcal B(\cP(\cC))$.
	Note that $\cP_p(\cC)=\cup_{\varepsilon>0} A_\varepsilon^\eta$  and so $\cP_p(\cC)\in \mathcal B(\cP(\cC))$. This means that $D_t^c=\{ \omega :\kappa_{X_{\cdot\wedge t},\cF_t}\in \cP_p(\cC)\}\in{\mathcal F}_t$ by the aforementioned measurability of $\kappa_{X_{\cdot\wedge t},\cF_t}$ and therefore $D_t$ is also in $\cF_t$. 
	
	Now assume for the sake of contradiction that $D_t$ has non-zero probability under $\P$. Then,
	\[
	\begin{split}
	\E[\sup_{0\leq s\leq t}|X_s|^p]=\E[\sup_{0\leq s\leq t}|X_s|^p(\1_{D_t}+\1_{{D_t}^c})]= &  \E[\E[\sup_{0\leq s\leq t}|X_s|^p|{\mathcal F}_t](\1_{D_t}+\1_{{D_t}^c})]\\
	=& \infty,
	\end{split}
	\]
	which is a contradiction.

	Finally, for some arbitrary but fixed distribution $\mu\in \cP_p(\cC)$ defining for all $t\in I$, $\mathscr L(X_t| \cF_t):=\kappa_{X_t,\cF_t}\1_{D^c_t}+\mu\1_{D_t}$ see that $\mathscr L(X_{\cdot\wedge t}| \cF_t)$ is an $\cF_t$-measurable $\cP_p(\cC)$ valued version of the regular conditional distribution of $X_{\cdot\wedge t}$ given ${\mathcal F}_t$ for each $t\in I$.
	%
	%

\end{proof}


\bibliography{FullBibliographyWRPH}  
%
%
%
%
%

\end{document}